\pdfoutput=1
\pdfmapfile{+didot.map}
\documentclass[english,11pt]{smfartreloaded}
\usepackage[utf8]{inputenc}
\usepackage[T1]{fontenc}
\usepackage[english,french]{babel}

\usepackage{ulem}
\newcommand{\varnothing}{\emptyset}
\newcommand{\coloneqq}{:=}
\usepackage{tikz}

\usepackage{xcolor,graphicx}
\usepackage[a4paper,vmargin={2cm,2cm},hmargin={2.5cm,2.5cm}]{geometry}

\usepackage[pdfpagemode=UseNone,bookmarksopen=false,colorlinks=true,urlcolor=blue,citecolor=blue,citebordercolor=blue,linkcolor=blue]{hyperref}

\usepackage{bbm}
\usepackage{enumitem}
	\renewcommand{\theenumi}{{(\roman{enumi})}}
	\renewcommand{\labelenumi}{\theenumi}
\usepackage[capitalize]{cleveref}

\newcommand{\Z}{\mathbb{Z}}
\renewcommand{\P}{\mathbb{P}}

\def\U{\mathbb{U}}
\def\N{\mathbb{N}}
\newcommand{\R}{\mathbb{R}}
\newcommand{\D}{\mathbb{D}}

\def\d{{\rm d}}

\renewcommand{\epsilon}{\varepsilon}
\newcommand\Es[1]{\mathbb{E}\left[#1\right]}
\renewcommand\Pr[1]{\mathbb{P}\left(#1\right)}

\newtheorem{theorem}{Theorem}[]

\newtheorem{proposition}[theorem]{Proposition}
\newtheorem{lemma}[theorem]{Lemma}
\newtheorem{corollary}[theorem]{Corollary}

\theoremstyle{definition}
\newtheorem{remark}[theorem]{Remark}  
\newtheorem{example}[theorem]{Example}

\def\llbracket{[\hspace{-.10em} [ }
\def\rrbracket{ ] \hspace{-.10em}]}

\def\build#1_#2^#3{\mathrel{
\mathop{\kern 0pt#1}\limits_{#2}^{#3}}}

\newcommand{\Xn}{X^{(n)}}
\newcommand{\Yn}{Y^{(n)}}
\newcommand{\Zn}{Z^{(n)}}
\newcommand{\Wn}{W^{(n)}}

\newcommand{\Ygn}{Y^{(n)}}
\newcommand{\Zgn}{\vec{Z}^{(n)}}
\newcommand{\Wgn}{\vec{W}^{(n)}}
\newcommand{\Bn}{B^{(n)}}

\newcommand{\Rn}{R^{(n)}}

\newcommand{\cA}{\mathcal{A}}
\newcommand{\Tc}{\mathcal{T}}

\newcommand{\Loop}{\mathsf{Loop}}
\newcommand{\rad}{\mathsf{rad}}

\newcommand{\q}{\mathsf{q}}
\newcommand{\m}{\mathbf{m}}

\newcommand\BGW{\textup{\textrm{BGW}}}
\def \W {\mathsf{W}}

\def \Mc {\mathcal{M}}
\def \Tn {\mathcal{T}_{n}}
\def \Tgn {\mathcal{T}_{\geq n}}

\title[Condensation in  Cauchy Bienaymé--Galton--Watson trees]{Condensation \\ in \\  critical Cauchy Bienaymé--Galton--Watson  trees}

\author{\FirstBigRestSmallUnUn{Igor} \FirstBigRestSmallUnUn{Kortchemski}}
\address{CNRS \& CMAP, \'Ecole polytechnique}
\email{igor.kortchemski@math.cnrs.fr}

\author{\FirstBigRestSmallUnUn{Loïc} \FirstBigRestSmallUnUn{Richier}}
\address{CMAP, \'Ecole polytechnique}
\email{loic.richier@polytechnique.edu}



\begin{document}
\maketitle

\begin{abstract}
We are interested in the structure of large  Bienaymé--Galton--Watson random trees whose offspring distribution is critical and falls within the domain of attraction of a stable law of index $\alpha=1$. In stark contrast to the case $\alpha \in (1,2]$, we show that a condensation phenomenon occurs: in such trees, one vertex with macroscopic degree emerges.  To this end, we establish limit theorems for centered downwards skip-free random walks whose steps are in the domain of attraction of a Cauchy distribution, when conditioned on a late entrance in the negative real line. These results are of independent interest. As an application, we study the geometry of the boundary of random planar maps in a specific regime (called non-generic of parameter $3/2$). This supports the conjecture that faces in Le Gall \& Miermont's  $3/2$-stable maps are self-avoiding.
\end{abstract}

%

\begin{figure}[h!]
\centering
\begin{minipage}{.46\textwidth}
  \centering
  \includegraphics[width=\linewidth]{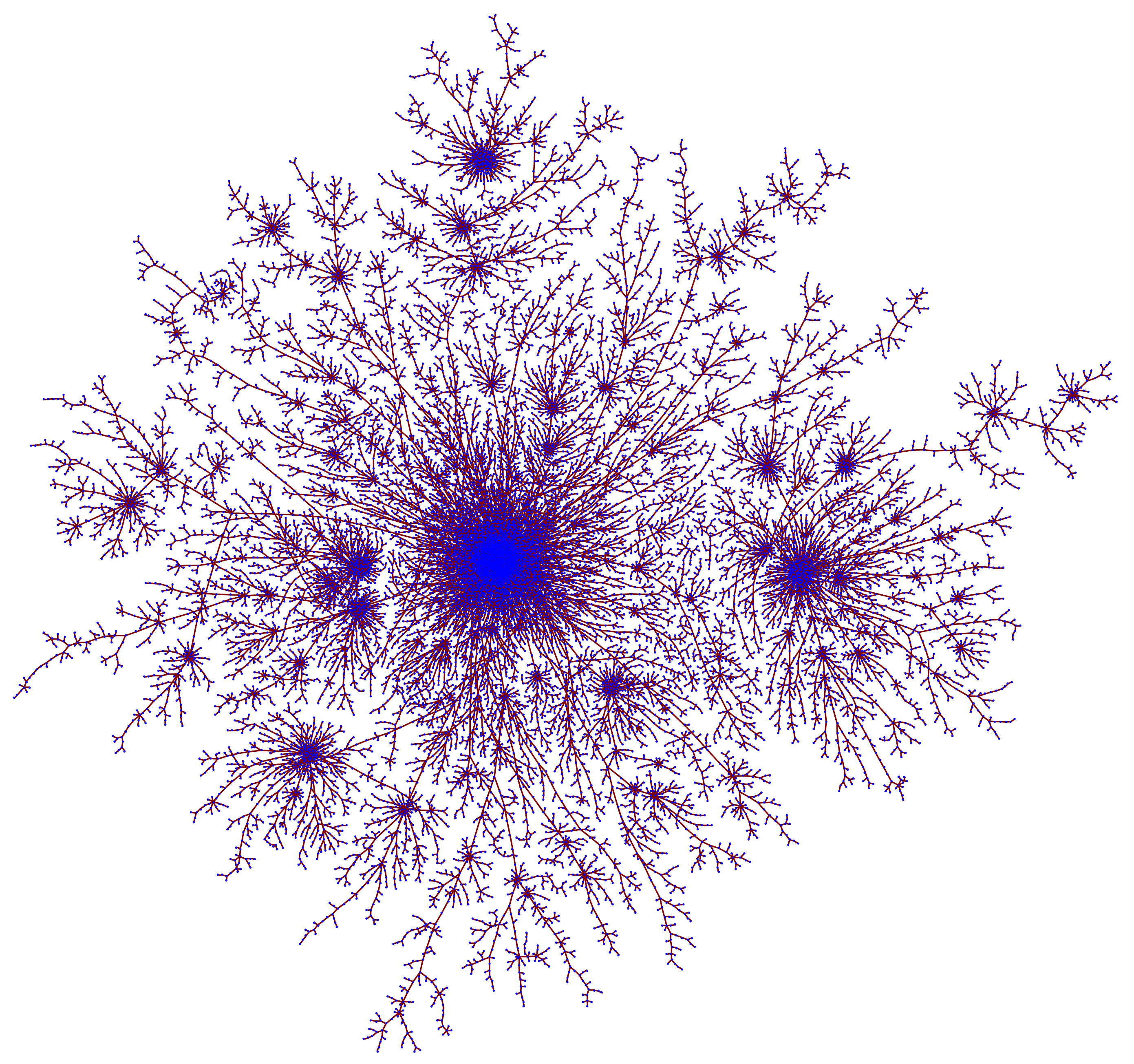}
  \label{fig:test1}
\end{minipage}
\hfill
\begin{minipage}{.48\textwidth} 
  \centering
  \includegraphics[width=\linewidth]{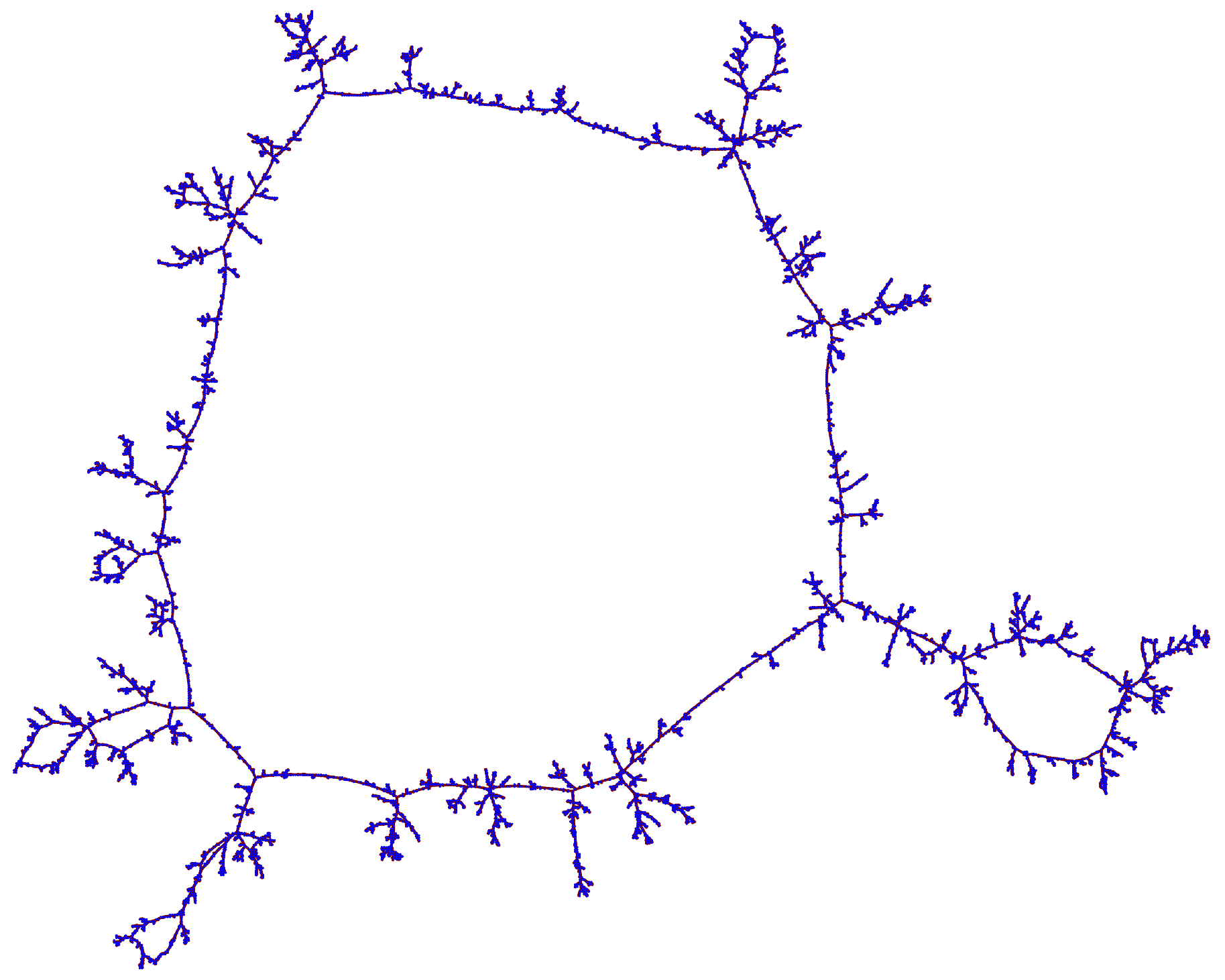}
  \label{fig:test2}
\end{minipage}
\caption{\label{fig:tree}Left: an embedding in the plane of a  Bienaymé--Galton--Watson tree with a critical  offspring distribution $\mu$ such that $\mu(k)= \frac{1}{3k^{2}\ln(k)^{2}}$ for $k \geq 3$, having $20000$ vertices (simulated using \cite{Dev12}).  Right: its associated looptree.}
\end{figure}

\pagebreak

\section{Introduction}
\label{sec:Intro}

\subsection{Context}

This work is concerned with the influence of the offspring distribution on the geometry of large Bienaymé--Galton--Watson (BGW) trees. The usual approach to understand the geometry of a BGW tree conditioned on having size $n$, that we denote by $ \mathcal{T}_{n}$, consists in studying the limit of $ \mathcal{T}_{n}$ as $n \rightarrow \infty$. There are essentially two notions of limits for random trees: the ``scaling'' limit framework (where one studies rescaled versions of the tree) and the ``local'' limit framework (where one looks at finite neighborhoods of a vertex).

\paragraph*{Limits for critical offspring distributions} The study of local limits of BGW trees with critical offspring distribution $\mu$ (i.e.\ with mean $m_{\mu}=1$) was initiated by Kesten in \cite{Kes86}. Assuming also that $\mu$  has finite variance, he proved that $\mathcal{T}_{n}$ (actually under a slightly different conditioning) converges locally in distribution as $n \to \infty$ to the so-called critical BGW tree conditioned to survive (which is a random locally finite tree with an infinite ``spine''). The same result was later established under a sole criticality assumption by Janson \cite{Jan12}. 

In the scaling limit setting, Aldous \cite{Ald93} showed that when $\mu$ has finite variance, the (rescaled) contour function of the tree converges in distribution to the Brownian excursion, which in turns codes the Brownian \textit{continuum random tree}. The second moment condition on $ \mu$ was later relaxed by Duquesne \cite {Du03} (see also \cite{Kor13}), who focused on  the case where $ \mu$ belongs to the domain of attraction of a stable law of index $ \alpha \in (1,2]$ (when $\mu$ has infinite variance, this means that $\mu([i,\infty))= {L(i)}/i^{\alpha}$ with $L$ a slowly varying function at $\infty$). He showed that the (rescaled) contour function of $  \mathcal{T}_{n}$ converges in distribution towards the normalized excursion of the $ \alpha$-stable height process, coding in turn the so-called $\alpha$-stable tree introduced in \cite{LGLJ98,DLG02}.

\paragraph*{Limits for subcritical offspring distributions} When the offspring distribution $\mu$ is subcritical (i.e.\ with mean $m_{\mu}<1$), the geometry of $\Tn$ is in general very different. Jonsson \& Stef\'ansson \cite{JS10} showed that if $\mu(i) \sim c/i^{\alpha+1}$ as $i \rightarrow \infty$ with $\alpha>1$, a condensation phenomenon occurs:  with probability tending to $1$ as $ n \rightarrow \infty$, the  maximal degree of $  \mathcal{T}_{n}$ is asymptotic to $(1- m_{\mu})n$. In addition, they showed that $ \mathcal{T}_{n}$ converges locally in distribution to a random tree that has a unique vertex of infinite degree (in sharp contrast with Kesten's tree).

These results were improved in \cite{Kor15}, which deals with the case where $\mu$ is subcritical and $\mu(i)=L(i)/i^{\alpha+1}$ with $L$ slowly varying at infinity and $\alpha>1$.
It was shown, roughly speaking, that $ \mathcal{T}_{n}$ can be constructed as a finite spine with height following a geometric random variable (with a finite number of BGW trees grafted on it), and approximately $(1-m_{\mu})n$ BGW trees grafted on the top of the spine. In some sense, the vertex with maximal degree of $ \mathcal{T}_{n}$ ``converges''  to the vertex of infinite degree in the local limit, so that this limit describes rather accurately the whole tree.

The behaviour of BGW trees when $\mu$ is subcritical and in the domain of attraction of a stable law is not known in full generality without regularity assumptions on $\mu(n)$. However, the geometry of $ \mathcal{T}_{ \geq n}$, which is the BGW tree under the weaker conditioning  to have size \textit{at least} $n$ has been described in \cite{KR18}, in view of applications to random planar maps in a case where regularity assumptions on $\mu(n)$ are unknown.

\subsubsection*{Critical Cauchy BGW trees.}
 The purpose of this work is to investigate a class of offspring distributions which has been left aside until now, namely offspring distributions which are critical and belong to the domain of attraction of a stable law of index $1$. We will prove that even though $ \mathcal{T}_{n}$ converges locally in distribution to the  critical $\BGW$ tree conditioned to survive (that has an infinite spine), a condensation phenomenon occurs. More precisely, with probability tending to $1$ as $n \rightarrow \infty$, the maximal degree in $\Tn$ dominates the others, but there are many vertices with degree of order $n$ up to a slowly varying function (in particular, this answers negatively Problem 19.30 in \cite{Jan12}, as we will see). This means that vertices with macroscopic degrees ``escape to infinity'' and disappear in the local limit. 

Although interesting for itself, this has applications to the study of the  boundary of non-generic critical Boltzmann maps with parameter $3/2$, as will be explained in Section \ref{sec:Maps}.

\medskip

Note that depending on $\mu$, Janson \cite{Jan12} classified in full generality the local limits of $ \mathcal{T}_{n}$. However, this local convergence is not sufficient to understand global properties of the tree.
For example, Janson \cite[Example 19.37]{Jan12} gives examples of BGW trees converging locally to the same limit, but, roughly speaking, such that in one case all vertices have degrees $o(n)$, and in the second case there are two vertices of degree  $n/3$.
In addition, outside of the class of critical offspring distributions in the domain of attraction of a stable law of index $\alpha \in (1,2]$, it is folklore that  the contour function of $ \mathcal{T}_{n}$ does not have non-trivial scaling limits. It is therefore natural to wonder whether one could still describe the global structure of $ \mathcal{T}_{n}$ outside of this class.

In the recent years, it has been realized that BGW trees in which a condensation phenomenon occurs code a variety of random combinatorial structures such as random planar maps \cite{Add15,JS15,Ric17}, outerplanar maps \cite{SS17}, supercritical percolation clusters of random triangulations \cite{CK15} or minimal factorizations \cite{FK17}. See \cite{Stu16} for a combinatorial  framework and further examples. These applications are one of the motivations for the study of the fine structure of such large conditioned BGW trees.

\subsection{Looptrees}

In order to study the condensation phenomenon, the notion of a \emph{looptree} will be useful. Following \cite{CK14b}, with every plane tree $ \tau$ we associate a graph denoted by $\Loop(\tau)$ and called a looptree. This graph has the same set of vertices as $\tau$, and for every vertices $u,v\in \tau$, there is an edge between $u$ and $v$ in $\Loop(\tau)$ if and only if $u$ and $v$ are consecutive children of the same parent in $\tau$, or if $v$ is the first or the last child of $u$ in $\tau$ (see Figure \ref{fig:loopintro} for an example). We view $\mathsf{Loop}( \tau)$ as a compact metric space by endowing its vertices with the graph distance.

 \begin{figure}[!h]
 \begin{center}
 \includegraphics[width=  \linewidth]{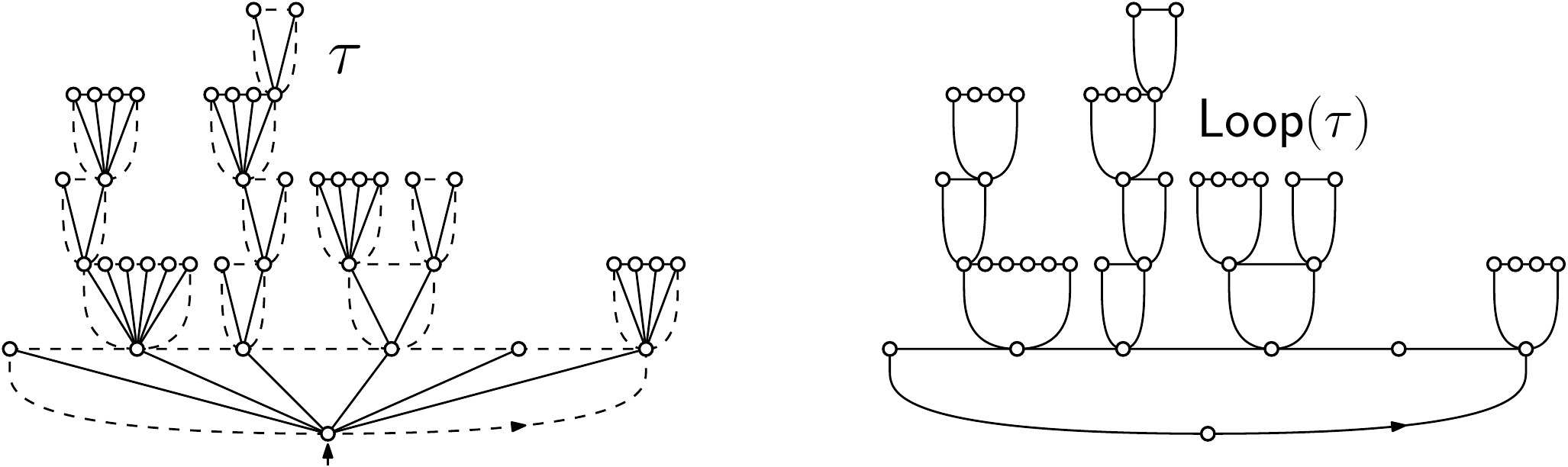}
 \caption{ \label{fig:loopintro}A plane tree $\tau$ and its associated looptree $ \mathsf{Loop}( \tau)$.}
 \end{center}
 \end{figure}
 
The notion of a looptree is very convenient to give a precise formulation of the condensation principle. Namely, we say that a sequence $(\tau_{n} : n \geq 1)$ of plane trees exhibits (global) condensation if there exists a sequence $\gamma_{n} \rightarrow \infty $ and a positive random variable $V$ such that the convergence
\begin{equation}
\label{eq:condensation}
\frac{1}{\gamma_{n}} \cdot \Loop(\tau_{n})  \quad \xrightarrow[n\to\infty]{(d)} \quad V \cdot \mathbb{S}_1
\end{equation}
holds in distribution with respect to the Gromov--Hausdorff topology (see \cite[Chapter 7.3]{burago_course_2001} for background), where for every $\lambda>0$ and every metric space $(E,d)$, $\lambda\cdot E$ stands for $(E,\lambda \cdot d)$ and $\mathbb{S}_{1}$ is the unit circle.

Since, intuitively speaking, $\Loop(\tau_n)$ encodes the structure of large degrees in $\tau_n$, the convergence  \eqref{eq:condensation} indeed tells that a unique vertex of macroscopic degree (of order $\gamma_{n}$) governs the structure of $\tau_{n}$.

Translated in terms of looptrees, the results of \cite{Kor15} indeed show that when $\mu$ is subcritical and $\mu(i)=L(i)/i^{\alpha+1}$ with $L$ slowly varying at infinity and $\alpha>1$, if $ \mathcal{T}_{n}$ is a $\BGW$ tree with offspring distribution $\mu$ conditioned on having $n$ vertices, then
\[ \frac{1}{n} \cdot \Loop( \mathcal{T}_{n})  \quad \xrightarrow[n\to\infty]{(d)} \quad  (1-m_{\mu}) \cdot \mathbb{S}_{1},\]
where $m_{\mu}$ is the mean of $\mu$. As we will see, condensation occurs for a $\BGW$ tree with \emph{critical} offspring distribution in the domain of attraction of a stable law of index $1$, but at a scale which is negligible compared to the total size of the tree.

\subsection{Framework and scaling constants}

Let $L$ be a slowly varying function at $\infty$ (see \cite{BGT89} for background on slowly varying functions). Throughout this work, we shall work with offspring distributions $\mu$ such that 
\begin{center}
\uwave{\hspace{\linewidth}}
\begin{equation}
   \label{eq:hypmu}
  \mu \text{ is critical} \qquad \text{and} \qquad  \mu([n,\infty))  \quad \mathop{\sim}_{n \rightarrow \infty} \quad  \frac{L(n)}{n}.
  \tag{$\textrm{H}_\mu$}
   \end{equation}
   \uwave{\hspace{\linewidth}}
\end{center}

\medskip
   
We now consider BGW trees with critical offspring distribution $\mu$ ($\BGW_\mu$ in short). Let us introduce some important scaling constants which will appear in the description of large $\BGW_{\mu}$ trees. To this end, we use a 
random variable $X$ with law given by $\P(X=i)=\mu(i+1)$ for $i \geq -1$ (observe that $X$ is centered since $\mu$ is critical). 
Let $(a_{n} : n \geq 1)$ and $(b_{n} : n \geq 1)$ be sequences such that
\begin{equation}
\label{eq:defanbn} n \P(X \geq a_{n})  \quad \xrightarrow[n\to\infty]{} \quad  1, \qquad b_{n}=n \Es{X \mathbbm{1}_{|X| \leq a_{n}}}.
\end{equation} 
The main reason why these scaling constants appear is the following: if $(X_{i}:i \geq 1)$ is a sequence of i.i.d.~random variables distributed as $X$, then the convergence
\[ \frac{ X_{1}+ \cdots+X_{n}-b_{n}}{a_{n}}  \quad \xrightarrow[n\to\infty]{(d)} \quad  \mathcal{C}_{1}\] 
holds in distribution, where $ \mathcal{C}_{1}$ is the random variable with Laplace transform given by $\Es{e^{-\lambda \mathcal{C}_{1} } }=e^{\lambda \ln(\lambda)}$ for $\lambda>0$ ($\mathcal{C}_1$ is an asymmetric Cauchy random variable with skewness $1$, see \cite[Chap. IX.8 and Eq.~(8.15) p315]{Fel71}).

It is well known that $a_{n}$ and $b_{n}$ are both regularly varying with index $1$, and that $b_{n} \rightarrow -\infty$ and $|b_{n}|/a_{n} \rightarrow \infty$ as $n \rightarrow \infty$. One can express an asymptotic equivalent of $b_{n}$ in terms of $L$, see~\eqref{eq:bn}.

\smallskip
 
For example, if $\mu(n) \sim \frac{c}{n^{2} \ln(n)^{2}}$, we have $a_{n} \sim  \frac{cn}{\ln(n)^2}$ and $ b_{n} \sim - \frac{cn}{\ln(n)}$ (see  Example \ref{ex:1}). We encourage the reader to keep in mind this example to feel the orders of magnitude involved in the limit theorems.
  
    \subsection{Local conditioning}

We start with the study of a $\BGW_{\mu}$ tree conditioned on having exactly $n$ vertices, which will be denoted by  $ \mathcal{T}_{  n}$.   As in the subcritical case considered in \cite{Kor15}, it is not clear how to analyze the behavior of $ \mathcal{T}_{n}$ under a sole assumption on $ \mu([n,\infty))$. For this reason, when studying this local conditioning, as in \cite{Kor15} we shall work under the stronger assumption that

\begin{center}
\uwave{\hspace{\linewidth}}
\begin{equation}
   \label{eq:hypmustar}
  \mu \text{ is critical} \qquad \text{and} \qquad  \mu(n)  \quad \mathop{\sim}_{n \rightarrow \infty} \quad  \frac{L(n)}{n^2}.
  \tag{$\textrm{H}_\mu^{\mathrm{loc}}$}
   \end{equation}\uwave{\hspace{\linewidth}}
\end{center} 

\medskip

Of course, this implies \eqref{eq:hypmu}, but the converse is not true. Then, denoting by $\Delta(\tau)$ the maximal degree of a tree $\tau$, our main result is the following.
       
\begin{theorem}
  \label{thm:loop-local}
  Assume that $\mu$ satisfies \eqref{eq:hypmustar}. Then the convergences 
   \begin{equation*}\frac{\Delta( \mathcal{T}_{  n})-|b_{n}| }{a_{n}}   \quad \xrightarrow[n\to\infty]{(d)} \quad \mathcal{C}_{1} \qquad \text{and}   \qquad \frac{1}{|b_{n}|} \cdot \Loop( \mathcal{T}_{n})    \quad \xrightarrow[n\to\infty]{(d)} \quad   \mathbb{S}_1
   \end{equation*} hold in distribution (with respect to the Gromov--Hausdorff topology for the second one).
   \end{theorem}
 
Therefore, condensation occurs at scale $|b_{n}|$ in $\Tn$. Observe that $|b_{n}|=o(n)$, in sharp contrast with the subcritical case where condensation occurs at scale $n$.  Moreover, since $a_{n}=o(|b_{n}|)$, $\Delta( \mathcal{T}_{  n})/|b_{n}| \rightarrow 1$ in probability (but the above result also gives the fluctuations of $\Delta( \mathcal{T}_{  n})$ around~$|b_{n}|$).

In order to establish \cref{thm:loop-local}, we will use the coding of $ \mathcal{T}_{n}$  by its {\L}ukasiewicz path, which is a centered random walk conditioned on a fixed entrance time in the negative real line. We show that, asymptotically, this conditioned random walk is well approximated in total variation by a simple explicit random trajectory $\Zn$ (\cref{thm:dTVLoc}). To this end, we adapt arguments of Armend\'ariz \& Loulakis \cite{AL11}. The process $\Zn$ is then constructed by relying on a path transformation of a random walk (the Vervaat transform, see  \cref{sec:excursioncond} for details).

This approximation has several interesting consequences. First, it allows us to establish that a ``one big jump'' principle occurs for $Z^{(n)}$ (\cref{thm:RW-local}) and for the {\L}ukasiewicz path of $\mathcal{T}_{n}$ (\cref{prop:luka-local}), which is a key step to prove \cref{thm:loop-local}. Second, it allows to determine the order of magnitude of the height $H^{\ast}_{n}$ of the vertex with maximal degree in $ \mathcal{T}_{n}$ (that is, its graph distance to the root vertex).

\begin{theorem}
  \label{thm:height-local}There is a slowly varying function $\Lambda$ with $\Lambda(x)\rightarrow \infty$ as $x\rightarrow \infty$ such that
  \[ \frac{H^{\ast}_{  n}}{\Lambda(n)}  \quad \xrightarrow[n\to\infty]{(d)} \quad \mathsf{Exp}(1),\]
where $\mathsf{Exp}(1)$ is an exponential random variable with parameter 1.
   \end{theorem} In the particular case $\mu(n) \sim \tfrac{c}{n^{2} \ln(n)^{2}}$, one can take $\Lambda(n)= \tfrac{\ln(n)}{c^{2}}$ (see \cref{lem:estimateslocal} and \cref{rem:constants}). In contrast with the subcritical  case, where the height of the vertex of maximal degree converges in distribution to a geometric random variable \cite[Theorem 2]{Kor15}, here $H^{\ast}_{n} \rightarrow \infty$ in probability. This is consistent with the fact that in the subcritical case, the local limit is a tree with one vertex of infinite degree, while in the critical case $\mathcal{T}_{n}$ converges in distribution for the local topology to a locally finite tree. Roughly speaking, vertices with large degrees in  $\mathcal{T}_{n}$ ``escape to infinity'' as $n \rightarrow \infty$.

\smallskip

Intuitively speaking, the approximation given by \cref{thm:dTVLoc} implies that the tree $ \mathcal{T}_{n}$ may be seen as a ``spine'' of height $H^{\ast}_{n}$; to its left and right are grafted independent $\BGW_{\mu}$ trees, and on the top of the spine is grafted a forest of $ \simeq |b_{n}|$  $\BGW_{\mu}$ trees. In particular, this description implies that while the maximal degree of  $ \mathcal{T}_{n}$ is of order $|b_{n}|$,  the next largest degrees of $ \mathcal{T}_{n}$ are of order $a_{n}$ in the following sense.

   \begin{theorem}
  \label{thm:degrees-local}
Let $(\Delta ^{(i)}_{  n} : i \geq 0)$ be the degrees of $ \mathcal{T}_{  n}$ ordered in decreasing order. Then the convergence
\[ \left( \frac{\Delta^{(0)}_{  n} }{|b_{n}|},\frac{\Delta^{(1)}_{  n} }{a_{n}},\frac{\Delta^{(2)}_{  n} }{a_{n}}, \ldots \right)\quad \xrightarrow[n\to\infty]{(d)} \quad (1, \Delta^{(1)}, \Delta^{(2)}, \ldots)\]
holds in distribution for finite dimensional marginals, where $(\Delta^{(i)} : i \geq 1)$ is the decreasing rearrangement of the second coordinates of the atoms of a Poisson measure on $[0,1] \times \R_{+}$ with intensity $\textrm{d}t \otimes \frac{\textrm{d}x}{x^{2}}$.
   \end{theorem}  
   
This result shows that there are  many vertices with degrees of order $n$ up to a slowly varying function. However, the maximum degree of $ \mathcal{T}_{n}$ is at a different scale from the others since  $a_{n}=o(|b_{n}|)$. In particular,  $\Delta_{n}^{(0)}/\Delta_{n}^{(1)} \rightarrow \infty$ in probability as $n \rightarrow \infty$. This also answers negatively  Problem 19.30 in \cite{Jan12} (in the case $\lambda=\nu=1$ in the latter reference), as by taking $h(n)= \sqrt{a_{n} |b_{n}|}$, we have $n \P(X \geq h(n)) \rightarrow 0$, but it is not true that $\Delta_{n}^{(1)} \leq  h(n)$ with high probability.

\subsection{Tail conditioning}

At this point, the reader may wonder if the results of the previous section hold under the more general assumption \eqref{eq:hypmu}. In this case, it is not known if the required estimates on random walks are still valid. For this reason, analogous results may be obtained at the cost of relaxing the conditioning on the total number of vertices of the tree. Namely, as in \cite{KR18}, we now deal with $ \mathcal{T}_{ \geq n}$, a $\BGW_{\mu}$ tree conditioned on having \emph{at least} $n$ vertices, where $\mu$ is a fixed offspring distribution satisfying \eqref{eq:hypmu}. A motivation for studying this conditioning is the application to large faces in random planar maps given in the last section, where we merely know that the assumption \eqref{eq:hypmu} is satisfied. 

\smallskip

Similarly to the previous setting, we use the coding of $ \mathcal{T}_{ \geq n}$  by its {\L}ukasiewicz path, which is then a centered random walk conditioned on a \emph{late entrance} in the negative axis. We show that, asymptotically, this conditioned random walk is well approximated in total variation by a simple explicit random trajectory $\Zgn$ (\cref{thm:dTVtail}). To this end, we rely on the strategy developed in \cite{KR18} and we obtain a new asymptotic equivalence on the tails of ladder times of random walks (\cref{prop:laddertimes}), which improves recent results of Berger \cite{Ber17} and is of independent interest. Even though the global strategy is similar to the local case, we emphasize that $\Zgn$ is of very different nature than the one we introduce in the local conditioning.

As we shall see, this approximation has several interesting consequences. First, it allows to establish that a ``one big jump'' principle occurs for $\Zgn$ (\cref{thm:RW-local}) and for the {\L}ukasiewicz path of $\Tgn$ (\cref{prop:luka-local}). It is interesting to note that  similar ``one big jump'' principles have been established for random walks with \emph{negative drift} by Durrett \cite{Dur80} in the case of jump distributions with finite variance and in \cite{KR18} in the case of jump distributions in the domain of attraction of a stable law of index in $(1,2]$. However, we deal here with \emph{centered} random walks. Second, it yields a decomposition of the tree $\Tgn$ which is very similar in spirit to that of the tree $\Tn$, except that the maximal degree in the tree remains random in the scaling limit. As a consequence, we obtain the following analogues of Theorems \ref{thm:loop-local}, \ref{thm:height-local} and \ref{thm:degrees-local}. Once Theorem \ref{thm:dTVtail} is established, their proofs are simple adaptations of those in the local conditioning setting, and will be less detailed. We start with the existence of a condensation phenomenon.
 
\begin{theorem}
  \label{thm:loop-tail}
    Let  $J$ be the real-valued random variable such that $\Pr{J \geq x}= 1/x$ for $x \geq 1$. Assume that $\mu$ satisfies \eqref{eq:hypmu}. Then the convergence
   \begin{equation}
   \label{eq:loop-tail}\frac{1}{|b_{n}|} \cdot \Loop( \mathcal{T}_{ \geq n})  \quad \xrightarrow[n\to\infty]{(d)} \quad J \cdot \mathbb{S}_1
   \end{equation}
   holds in distribution
with respect to the Gromov--Hausdorff topology.
   \end{theorem} The height $H^{\ast}_{ \geq n}$ of the vertex with maximal degree in $ \mathcal{T}_{ \geq n}$ turns out to have the same order of magnitude as $H^{\ast}_{n}$.
 
   \begin{theorem}
  \label{thm:height-tail}
 Let $\Lambda$ be a slowly varying function such that the conclusion of \cref{thm:height-local} holds. Then
\[ \frac{H^{\ast}_{ \geq n}}{\Lambda(n)}  \quad \xrightarrow[n\to\infty]{(d)} \quad \mathsf{Exp}(1),\]
where $\mathsf{Exp}(1)$ is an exponential random variable with parameter 1.
   \end{theorem}  
   
 Finally, the distribution of the sequence of higher degrees in $\Tgn$ can also be studied.

       \begin{theorem}
  \label{thm:degrees-tail}
Let $(\Delta ^{(i)}_{ \geq n} : i \geq 0)$ be the degrees of $ \mathcal{T}_{ \geq n}$ ordered in decreasing order. Then the convergence
\[ \left( \frac{\Delta^{(0)}_{ \geq n} }{|b_{n}|},\frac{\Delta^{(1)}_{ \geq n} }{a_{n}},\frac{\Delta^{(2)}_{ \geq n} }{a_{n}}, \ldots \right) \quad \xrightarrow[n\to\infty]{(d)} \quad (\Delta^{(0)}, \Delta^{(1)}, \Delta ^{(2)}, \ldots)\]holds in distribution for finite dimensional marginals, where $\Delta^{(0)} \overset{(d)}{=} J$, and conditionally given $\Delta^{(0)}=a$, $(\Delta^{(i)} : i \geq 1)$ is the decreasing rearrangement of the second coordinates of the atoms of a Poisson measure on $[0,a] \times \R_{+}$ with intensity $\textrm{d}t \otimes \frac{\textrm{d}x}{x^{2}}$.
   \end{theorem}  
   
\begin{remark} The results of this paper deal with the large scale geometry of $\BGW$ trees whose offspring distribution is in the domain  of attraction of a stable law with index $1$ and is \textit{critical}. However, in the \textit{subcritical} case, one can prove that a condensation phenomenon occurs at a scale which is the total size of the tree by a simple adaptation of the arguments developed in \cite{Kor15,KR18} by using \cite{Ber17}. Moreover, in the \emph{supercritical} case, the Brownian CRT appears as the scaling limit in virtue of the classical ``exponential tilting'' technique \cite{Ken75}.
\end{remark}

\paragraph*{Acknowledgments} We are grateful to Quentin Berger and Grégory Miermont for stimulating discussions.  We acknowledge  partial support from  Agence Nationale de la Recherche, grant number ANR-14-CE25-0014 (ANR GRAAL). I.K. acknowledges  partial support from the City of Paris,  grant “Emergences Paris 2013, Combinatoire à Paris”. Finally, we thank an anonymous referee for correcting an inaccuracy in the proof of Corollary 14 (i) and for many useful suggestions.

\tableofcontents

\section{Bienaymé--Galton--Watson trees}
\label{sec:trees}

\subsection{Plane trees}
\label{sssec:planetrees}

Let us define plane trees according to Neveu's formalism \cite{Nev86}. Let $\N = \{1, 2, \dots\}$ be the set of positive integers, and consider the set of labels $\U = \bigcup_{n \ge 0} \N^n$ (where by convention $\N^0 = \{\varnothing\}$). For every $v = (v_1, \dots, v_n) \in \U$, the length of $v$ is $|v| = n$. We endow $\U$ with the lexicographical order, denoted by $\prec$. 

Then, a (locally finite) \emph{plane tree} is a nonempty subset $\tau \subset \U$ satisfying the following conditions. First, $\varnothing \in \tau$ ($\varnothing$ is called the \textit{root vertex} of the tree). Second, if $v=(v_1,\ldots,v_n) \in \tau$ with $n \ge 1$, then $(v_1,\ldots,v_{n-1}) \in \tau$ ($(v_1,\ldots,v_{n-1})$ is called the \textit{parent} of $v$ in $\tau$). Finally, if $v \in \tau$, then there exists an integer $k_v(\tau) \ge 0$ such that $(v_1,\ldots,v_n,i) \in \tau$ if and only if $1 \le i \le k_v(\tau)$ ($k_v(\tau)$ is the \textit{number of children} of $v$ in $\tau$). The plane tree $\tau$ may be seen as a genealogical tree in which the individuals are the vertices $v\in\tau$.

Let us introduce some useful notation. For $v,w \in \tau$, we let  $\llbracket v, w \rrbracket$ be the vertices belonging to the shortest path from $v$ to $w$ in $\tau$. Accordingly, we use $\llbracket v, w \llbracket$ for the same set, excluding $w$. We also let $|\tau|$ be the total number of vertices (that is, the size) of the plane tree $\tau$. 

\subsection{Bienaymé--Galton--Watson trees and their codings}
\label{sssec:coding}
Let $\mu$ be a  probability measure on $\Z_{\geq 0}$, that we call the \emph{offspring distribution}. We assume that $\mu(0) > 0$ and $\mu(0)+\mu(1)<1$ in order to avoid trivial cases. We also make the fundamental assumption that $\mu$ is \textit{critical}, meaning that it has mean $m_{\mu}:=\sum_{i \geq 0} i \mu (i) = 1$. The Bienaymé--Galton--Watson ($\BGW$) measure with offspring distribution $\mu$ is the probability measure $\BGW_\mu$  on plane trees that is characterized by
\begin{equation}\label{eq:def_GW}
\BGW_\mu(\tau) = \prod_{u \in \tau} \mu(k_u(\tau))
\end{equation}
for every finite plane tree $\tau$ (see~\cite[Prop.~1.4]{LG05}). 

Let $\tau$ be a plane tree whose vertices listed in lexicographical order are $\varnothing=u_{0}\prec u_{1}\prec \cdots \prec u_{|\tau|-1}$.  The \textit{{\L}ukasiewicz path} $ \W(\tau)=( \W_n(\tau) : 0 \leq n < |\tau|)$ of $\tau$ is the path defined by $ \W_0(\tau)=0$, and $ \W_{n+1}(\tau)= \W_{n}(\tau)+k_{u_n}(\tau)-1$ for every $0 \leq n < |\tau|$.
For technical reasons, we let $\W_n(\tau)=0$ for $ n> |\tau|$ or $n<0$.

The following result relates the {\L}ukasiewicz path of a $\BGW$ tree to a random walk (see \cite[Proposition 1.5]{LG05} for a proof). Let $(Y_{i} : i \geq 1)$ be a sequence of i.i.d.~real valued random variables with law given by $\P(Y_{1}=i)=\mu(i+1)$ for $i \geq -1$ and set  $\zeta= \inf \{i \geq 1 : {Y}_{1}+Y_{2}+ \cdots+Y_{i}<0\}$.

\begin{proposition}
\label{prop:GWRW}
Let $ \mathcal{T}$ be a tree with law $\BGW_{\mu}$. Then
\[(\W_{0}( \mathcal{T}), \W_{1}( \mathcal{T}), \ldots, \W_{| \mathcal{T}| }( \mathcal{T}))  \quad \mathop{=}^{(d)} \quad ({Y}_{0},{Y}_{1}, \ldots, Y_{1}+Y_{2}+ \cdots+{Y}_{\zeta}).\]
\end{proposition}

\subsection{Tail bound for the height of a critical Cauchy BGW tree}

The following preliminary lemma gives a rough estimate on the tail of the height of a $\BGW_\mu$ tree  when  $\mu$ satisfies \eqref{eq:hypmu}. Its proof may be skipped in a first reading.
 
 For every plane tree $\tau$, we let $\mathbf{H}(\tau):= \sup\{ \vert v \vert : v \in \tau\}$ be its total height.

\begin{lemma}\label{lem:HeightBGW}
	Under \eqref{eq:hypmu}, if $ \mathcal{T} $ is a $\BGW_\mu$ tree, we have $n\Pr{\mathbf{H}( \mathcal{T} ) \geq n} \rightarrow 0$ as $n \rightarrow \infty$.
\end{lemma}

\begin{proof} The idea of the proof is to dominate $\Pr{\mathbf{H}(\mathcal{T}) \geq n} $ by a similar quantity for an offspring distribution that is critical and belongs to the domain of attraction of a stable law with index strictly between $1$ and $2$. Indeed, in that case, estimates for the height of the tree are known by \cite{Sla68}.

Set $Q_n \coloneqq \Pr{\mathbf{H}(\mathcal{T})\geq n}$ for every $n\geq 0$. By conditioning with respect to the degree of the root vertex, we get that $(Q_n : n\geq 0)$ is the solution of the equation
\begin{equation}\label{eqn:recurrence}
\left\lbrace
\begin{array}{l}
Q_0= {1} \\
1-Q_{n+1}=G_\mu(1-Q_n), \qquad n \in \Z_{\geq 0} \\
\end{array}\right.,
\end{equation} where $G_\mu(s)=\sum_{i=0}^{\infty}\mu(i)s^{i}$ stands for the generating function of $\mu$. Observe that
	\[\sum_{k=n}^\infty k\mu(k)=\sum_{k=n}^\infty\mu([k,\infty))  + n \mu([n,\infty)),\]
	and recall that $\mu([n,\infty)) \sim L(n)/n$ with $L$ slowly varying. Setting $\ell^*(n)\coloneqq\sum_{k=n}^\infty \frac{L(k)}{k}$, we have $\ell^*(n)/L(n)\rightarrow 0$ when $n \rightarrow \infty$ by \cite[Proposition 1.5.9a]{BGT89}. We can then apply Karamata's Abelian theorem \cite[Theorem 8.1.6]{BGT89} to write 
	 \[G_\mu(s)=s+(1-s) h(1-s), \qquad s \in [0,1)\]
	 with $h$ slowly varying at $0$ and $h(x) \sim \ell^*({1/x})$ as $x \rightarrow 0$.  Thus,  \eqref{eqn:recurrence} may be rewritten as $Q_{n+1}=Q_n(1-h(Q_n))$ for $n\geq 0$. {We now let $\rho$ be an offspring distribution whose generating function is given by $G_\rho(s)=s+\frac{1}{2}(1-s)^{3/2}$. (We could define $\rho$ with any exponent $\beta \in (1,2)$ instead of $3/2$, but this will suffice for our purpose).}
	
	 If $\widehat{Q}_n$ denotes the probability that a $\BGW_\rho$ tree has height at least $n$, we similarly have $\widehat{Q}_{n+1}=\widehat{Q}_n(1- \frac{1}{2}\widehat{Q}_n^{1/2})$ for every $n\geq 0$. We introduce the functions
	\[f(x)=x(1-h(x)) \quad \text{and} \quad \widehat{f}(x)=x \left( 1- \frac{1}{2}x^{1/2}\right), \qquad  x \in [0,1].\] But since  $f(x)=1-G_{\mu}(1-x)$, $f$ is increasing. Moreover, $h$ is slowly varying at $0$ so by Potter's bound (see e.g.~\cite[Theorem 1.5.6]{BGT89}), there exists $A>0$ such that 
	\begin{equation}\label{eqn:boundFunc}
		f(x) \leq \widehat{f}(x), \qquad \forall\ x\in [0,A].
	\end{equation} Moreover, since $\widehat{Q}_n$ is decreasing and vanishes when $n$ goes to infinity, there exists $\widehat{N}\geq 1$ such that $\widehat{Q}_n < A$ for every $n \geq \widehat{N}$. Similarly, there exists $N\geq 1$ such that $Q_n\leq \widehat{Q}_{\widehat{N}}$ for every $n\geq N$ because $\widehat{Q}_{\widehat{N}}>0$. We now claim that 
	\[Q_{N+n}\leq \widehat{Q}_{\widehat{N}+n}, \qquad \text{ for every } n \in \Z_{\geq 0}.\] Let us prove this assertion by induction. The claim is clear for $n=0$, and then we have for $n\geq 0$
	\[Q_{N+n+1}=f\left(Q_{N+n}\right)\leq f\left(\widehat{Q}_{\widehat{N}+n}\right)\leq  \widehat{f}\left(\widehat{Q}_{\widehat{N}+n}\right) =  \widehat{Q}_{\widehat{N}+n+1},\] where we used the fact that $Q_{N+n}\leq \widehat{Q}_{\widehat{N}+n}\leq A$, that $f$ is increasing as well as \eqref{eqn:boundFunc}. By \cite[Lemma 2]{Sla68},
	\[\widehat{Q}_{n}  \quad \mathop{\sim}_{n \rightarrow \infty} \quad \frac{C}{n^{2}}\]
	for a certain constant $C>0$. This implies that $n \cdot Q_{n} \rightarrow 0$, and completes the proof. \end{proof}

\begin{remark}
By adjusting the choice of $\rho$ in the previous proof, it is a simple matter to show that for every $c>0$,  $n^{c} \cdot \Pr{\mathbf{H}( \mathcal{T} ) \geq n} \rightarrow 0$.  When $\ell^{*}(x)=o(1/\ln(x))$,  \cite{Sze76} gives an asymptotic equivalent for $\Pr{\mathbf{H}( \mathcal{T} )  \geq n}$ (see also \cite{VW07}). However, in general, there is no known asymptotic equivalent for $\Pr{\mathbf{H}( \mathcal{T} ) \geq n} $. 
\end{remark}

\section{Estimates for Cauchy random walks}\label{sec:Estimates}

The strategy of this paper is based on the study of $\BGW$ trees conditioned to survive \emph{via} their {\L}ukasiewicz paths. The statement of Proposition \ref{prop:GWRW} entails that under \eqref{eq:hypmu}, the {\L}ukasiewicz path of a $\BGW_\mu$ tree is a (killed) random walk on $\Z$ whose increment $X$ satisfies the following assumptions:
\begin{center}
\uwave{\hspace{\linewidth}}
\begin{equation}
   \label{eq:hypX}
  \Es{X}=0, \qquad \Pr{X \geq x} \quad \mathop{\sim}_{x \rightarrow \infty} \quad \frac{L(x)}{x} \qquad \text{and} \qquad  \Pr{X < -1}=0.
  \tag{$\textrm{H}_X$}
   \end{equation}
   \uwave{\hspace{\linewidth}}
\end{center}
\medskip

Here, we recall that $L$ is a slowly varying function. Then, we let $(X_i : i\geq 1)$ be a sequence of i.i.d.\ random variables distributed as $X$. We put $W_{0}=0$, $W_{n}=X_{1}+\cdots+X_{n}$ for every $n \geq 1$ and also let $W_{i}=0$ for $i<0$ by convention.

The goal of this section is to derive estimates (that are of independent interest) on the random walk $W$, that will be the key ingredients in the proofs of our main results.

\subsection{Entrance time and weak ladder times}
\label{sec:EstimatesGeneralSetting}

Recall that $(a_{n} : n \geq 1)$ and $(b_{n}: n \geq 1)$ are sequences such that 
 \begin{equation*}
 n \P(X \geq a_{n})\quad \xrightarrow[n\to\infty]{} \quad 1 \qquad \text{and} \qquad b_{n}=n \Es{X \mathbbm{1}_{|X| \leq a_{n}}}
 \end{equation*} and that
\begin{equation*}
\frac{W_{n}-b_{n}}{a_{n}}  \quad \xrightarrow[n\to\infty]{(d)} \quad  \mathcal{C}_{1},
\end{equation*} in distribution, where $ \mathcal{C}_{1}$ is an asymmetric Cauchy variable with skewness 1. 

One can express an asymptotic equivalent of $b_{n}$ in terms of $L$. Indeed, let $\ell^{*}$ be the function \[\ell^{*}(n)\coloneqq \sum_{k= n}^{\infty} \frac{L(k)}{k}, \quad n \in \N.\] By \cite[Proposition 1.5.9a]{BGT89}, $\ell^{*}$ is slowly varying, and as $n \rightarrow \infty$, $\ell^{*}(n) \rightarrow 0$ and $\ell^{*}(n)/L(n) \rightarrow \infty$. Then (see \cite[Lemma 7.3 \& Lemma 4.3]{Ber17}) we have
\begin{equation}
\label{eq:bn}b_{n} \quad \mathop{\sim}_{n \rightarrow \infty} \quad -n \ell^{*}(a_{n}) \quad \mathop{\sim}_{n \rightarrow \infty} \quad - n \ell^{*}(|b_{n}|).
\end{equation}
 It is  important to note that
\[b_{n}  \quad \xrightarrow[n\to\infty]{}  \quad  -\infty.\]

The above one-dimensional convergence can be improved to a functional convergence as follows (by e.g.~\cite[Theorem 16.14]{Kal02}). If $\D(\R_{+},\R)$ denotes the space of real-valued càdlàg functions on $\R_{+}$ equipped with the Skorokhod $J_{1}$ topology (see Chapter VI in \cite{JS03} for background), the convergence
\begin{equation}
\label{eq:cvfonctionnelle} \left( \frac{W_{\lfloor nt \rfloor}- b_{n} t}{a_{n}} : t \geq 0 \right)  \quad \xrightarrow[n\to\infty]{(d)}  \quad (\mathcal{C}_{t} : t \geq 0)
\end{equation} holds in distribution in $\D(\R_{+},\R)$,
where $\mathcal{C}$ is a totally asymmetric Cauchy process characterized by $\Es{e^{-\lambda \mathcal{C}_{1} } }=e^{\lambda \ln(\lambda)}$ for $\lambda>0$.

\medskip

The first quantity of interest is the distribution of the first entrance time $\zeta$ of the random walk $W$ into the negative half-line,
\[\zeta= \inf \{i \geq 1 : W_{i}<0\}.\] Then, we will consider the sequence $(T_{i}: i \geq 0)$ of (weak) ladder times of $(W_{n}: n \geq 1)$. That is, $T_{0}=0$, and, for $i \geq 1$,
\[T_{i+1}=\inf \{j > T_{i}: W_{j} \geq W_{T_{i}}\}.\] 
We say that $j \geq 0$ is a weak ladder time of $W$ if there exists $i \geq 0$ such that $j=T_{i}$.  We let $I_{n}$ be the last weak ladder time of $(W_{i}: 0 \leq i \leq n)$, that is, $I_n=\max\{0 \leq j \leq n : \exists \ i \geq 0 : T_i=j\}$.

\begin{lemma}
\label{lem:In}
For every $0 \leq j \leq n$, we have
\begin{equation}
\label{eq:In}
\P(I_{n}=j)= \P( \exists \ i \geq 0 : T_i=j) \P(T_{1}>n-j)=\P(\zeta > j)\P(T_{1}>n-j).
\end{equation}
\end{lemma}

\begin{proof}
The first equality follows from the Markov property of the random walk at time $j$.

For the second equality,  observe that saying that $j$ is a weak ladder time for $(W_{i}: 0 \leq i \leq j)$ is equivalent to saying that the random walk $(W^{[j]}_{i}: 0 \leq i \leq j)$ defined by \[W^{[j]}_{i}=W_{j}-W_{j-i}, \quad 0 \leq i \leq j\] satisfies $W^{[j]}_{i} \geq 0$ for every $0 \leq i \leq j$ (see Figure \ref{fig:timereversal}). Since $(W_{i}: 0 \leq i \leq j)$ and 
 $(W^{[j]}_{i}: 0 \leq i \leq j)$ have the same distribution, this implies that $\P(\exists \ i \geq 0 : T_i=j) =\P(\zeta > j)$ and  completes the proof.
\end{proof}

\begin{figure}[!ht]
\begin{scriptsize}
\begin{center}
\begin{tikzpicture}[scale=0.7]
\begin{scope}[shift={(0,-2.5)}]
\draw[thin, ->]	(0,0) -- (8.5,0);
\draw[thin, ->]	(0,-2) -- (0,2.5);
\foreach \x in {-2, -1, 1,2}
	\draw[dotted]	(0,\x) -- (8.5,\x);
\foreach \x in {-2, -1, 0, 1,2}
	\draw (.1,\x)--(-.1,\x)	(0,\x) node[left] {$\x$};
\coordinate (0) at (0, 0);
\coordinate (1) at (1, -1);
\coordinate (2) at (2, -2);
\coordinate (3) at (3, 2);
\coordinate (4) at (4, 1);
\coordinate (5) at (5, 1);
\coordinate (6) at (6, 2);
\coordinate (7) at (7, 1);
\coordinate (8) at (8, 2);

\newcommand{\lastx}{0}
\foreach \x [remember=\x as \lastx] in {1, 2, 3, ..., 8} \draw (\lastx) -- (\x);

\foreach \x in {2, 4, ..., 8}
	\draw (\x,.1)--(\x,-.1)	(\x,0) node[below] {$\x$};
	
\foreach \x in {0, 1, 2, 3, ..., 8} \draw [fill=black] (\x)	circle (2pt);
\end{scope}
\end{tikzpicture}
\qquad 
\begin{tikzpicture}[scale=0.7]
\begin{scope}[shift={(0,-2.5)}]
\draw[thin, ->]	(0,0) -- (8.5,0);
\draw[thin, ->]	(0,0) -- (0,4.5);
\foreach \x in {1,2,3,4}
	\draw[dotted]	(0,\x) -- (8.5,\x);
\foreach \x in { 0, 1,2,3,4}
	\draw (.1,\x)--(-.1,\x)	(0,\x) node[left] {$\x$};
\coordinate (0) at (0, 0);
\coordinate (1) at (1, 1);
\coordinate (2) at (2, 0);
\coordinate (3) at (3, 1);
\coordinate (4) at (4, 1);
\coordinate (5) at (5, 0);
\coordinate (6) at (6, 4);
\coordinate (7) at (7, 3);
\coordinate (8) at (8, 2);

\newcommand{\lastx}{0}
\foreach \x [remember=\x as \lastx] in {1, 2, 3, ..., 8} \draw (\lastx) -- (\x);

\foreach \x in {2, 4, ..., 8}
	\draw (\x,.1)--(\x,-.1)	(\x,0) node[below] {$\x$};
	
\foreach \x in {0, 1, 2, 3, ..., 8} \draw [fill=black] (\x)	circle (2pt);
\end{scope}
\end{tikzpicture}
\end{center}
\end{scriptsize}
\caption{\label{fig:timereversal}Left: an example of $(W_{i}: 0 \leq i \leq 8)$ such that $8$ is a weak ladder time; Right: its associated time-reversed path $(W^{[8]}_{i}: 0 \leq i \leq 8)$ (obtained by reading the jumps from right to left).}
\end{figure}
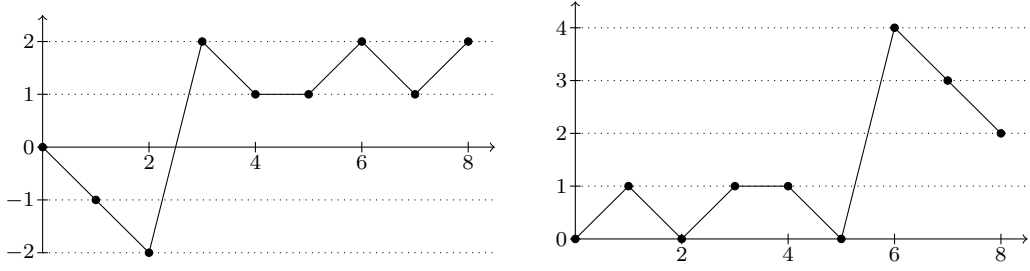

The estimates of \cref{prop:laddertimes} and \cref{prop:In} below will play an important role in the following. The first one extends \cite[Theorem 3.4]{Ber17} when there is no analyticity assumption on $L$.

\begin{proposition}
\label{prop:laddertimes}
There exists an increasing slowly varying function $\Lambda$ such that the following assertions hold.
\begin{enumerate}
\item We have
\[\P(\zeta>n)  \quad \mathop{\sim}_{n \rightarrow \infty} \quad   \frac{L(|{b}_{n}|)}{|b_{n}|} \cdot \Lambda(n) \qquad \text{and} \qquad \P(T_{1}>n)  \quad \mathop{\sim}_{n \rightarrow \infty} \quad  \frac{1}{ \Lambda(n)}.\]
\item We have
\[\sum_{k=0}^{n}\P(\zeta>k)  \quad \mathop{\sim}_{n \rightarrow \infty} \quad \Lambda(n).\]
\end{enumerate}
\end{proposition}

The key point is that the same slowly varying function $\Lambda$ appears in both asymptotic estimates. Its proof is based on a recent estimate of $\P(W_{n} \geq 0)$ due to Berger \cite{Ber17}.

\begin{proof}[Proof of \cref{prop:laddertimes}]
Our main input is the following estimate of \cite[Lemma 7.3]{Ber17}:
\begin{equation}
\label{eq:estimate1}\P(W_{n} \geq 0)  \quad \mathop{\sim}_{n \rightarrow \infty} \quad  \frac{L(|b_{n}|)}{\ell^{*}(|b_{n}|)}.
\end{equation}

Now, define the function $\Lambda$ by
\[\Lambda \left(  \frac{1}{1-s} \right) \coloneqq \exp \left( \sum_{k=1}^{\infty} \frac{\P(W_{k} \geq 0)}{k}s^{k} \right), \qquad s \in [0,1),\]
and observe that $\Lambda$ is increasing.

As $n \rightarrow \infty$, since $\P(W_{n} \geq 0) \rightarrow 0$, we also have $ \frac{1}{n} \sum_{k=1}^{n}\P(W_{k} \geq 0) \rightarrow 0$, so that by \cite[Lemma 1]{Rog71}, $\Lambda$ is slowly varying. Note that by \cite[Theorem 1.8.2]{BGT89} instead of working with regularly varying sequences $(\bullet_{n} : n \geq 1)$, we may work with infinitely differentiable functions $(\bullet_{u} : u \geq 0)$ (with $\bullet \in \{\ell^{*},L,\Lambda,a,b\}$).

For the first assertion, by the Wiener-Hopf factorization (see Theorem 4 in \cite[XII.7]{Fel71}):
\[p(s) \coloneqq \sum_{n=0}^{\infty} \P(\zeta>n) s^{n}= \exp \left( \sum_{k=1}^{\infty} \frac{\P(W_{k} \geq 0)}{k}s^{k} \right) = \Lambda \left(  \frac{1}{1-s} \right), \qquad s \in [0,1).\]
 Then, by \cite[Eq.~(7.23)]{Ber17}, we have
\[\sum_{k=1}^{\infty} \P(W_{k} \geq 0) s^{k}  \quad \mathop{\sim}_{s \uparrow 1} \quad   \frac{1}{1-s} \frac{L(|{b}_{1/(1-s)}|)}{\ell^{*}(|{b}_{1/(1-s)}|)},\]
so that
\[p'(s)  \quad \mathop{\sim}_{s \uparrow 1} \quad   \frac{1}{1-s} \frac{L(|{b}_{1/(1-s)}|)}{\ell^{*}(|{b}_{1/(1-s)}|)} \Lambda \left(  \frac{1}{1-s} \right).\]
In particular, setting $\hat{p}(u)=p(e^{-u})$ for $u>0$, we have
\begin{equation}
\label{eq:hatp}
\hat{p}'(u) \quad \mathop{\sim}_{u \downarrow 0} \quad   - \frac{1}{u} \frac{L(|{b}_{1/u}|)}{\ell^{*}(|{b}_{1/u}|)} \Lambda \left(  \frac{1}{u} \right).
\end{equation}
We now claim that it is enough to check that for every fixed $c>0$,
\begin{equation}
\label{eq:cvhat} { \hat{p} \left( \frac{1}{c t} \right) -  \hat{p} \left( \frac{1}{ t} \right)}   \quad \mathop{\sim}_{t \rightarrow \infty} \quad  \ln(c) \cdot \frac{L(|b_{t}|)}{\ell^{*}(|b_{t}|)}\Lambda(t).
\end{equation}
Indeed, by de Haan's monotone density theorem (see e.g.~\cite[Theorem 3.6.8 and 3.7.2]{BGT89}), this will imply that $\P(\zeta>n) \sim  \frac{L(|{b}_{n}|)}{|b_{n}|} \Lambda(n)$ since $|b_{n}|\sim n\ell^{*}(|b_{n}|)$. 
To establish \eqref{eq:cvhat}, write
\begin{align*}
  \hat{p} \left( \frac{1}{c t} \right) -  \hat{p} \left( \frac{1}{ t} \right) &= \int_{1/t}^{1/(ct)} \hat{p}'(u) {\d}u =  \int_{1}^{1/c} \hat{p}' \left( \frac{x}{t}\right) \frac{{\d}x}{t}.
 \end{align*}
 For a slowly varying function $\ell$, the convergence ${\ell(ax)}/{\ell(x)} \rightarrow 1$ holds uniformly for $a$ in compact subsets of $\R_{+}^{*}$ when $x \rightarrow \infty$ (see \cite[Theorem 1.5.2]{BGT89}), so by \eqref{eq:hatp} we have
 \[ \frac{1}{t}\hat{p}' \left( \frac{x}{t}\right)  \quad \mathop{\sim}_{t \rightarrow \infty} \quad- \frac{1}{x} \frac{L(|b_{t}|)}{\ell^{*}(|b_{t}|)}\Lambda(t),\]
 uniformly in $\min(1,1/c) \leq x \leq \max(1,1/c)$. Thus
 \[\hat{p} \left( \frac{1}{c t} \right) -  \hat{p} \left( \frac{1}{ t} \right)  \quad \mathop{\sim}_{t \rightarrow \infty} \quad \frac{L(|b_{t}|)}{\ell^{*}(|b_{t}|)}\Lambda(t) \int_{1/c}^{1} \frac{{\d}x}{x}= \ln(c) \cdot \frac{L(|b_{t}|)}{\ell^{*}(|b_{t}|)}\Lambda(t).\]
 This establishes \eqref{eq:cvhat}. Note also that since $p(s)= \Lambda(1/(1-s))$, Karamata's Tauberian theorem for power series \cite[Corollary 1.7.3]{BGT89} readily implies assertion $(ii)$. 
 
 We now turn to the behavior of $\P(T_{1}>n)$ as $n \rightarrow \infty$. Once again, by the Wiener-Hopf factorization \cite[XII.7, Theorem 1]{Fel71}, we have
 \[1-\sum_{n=0}^{\infty} \P(T_1=n) s^{n}= \exp \left( -\sum_{k=1}^{\infty} \frac{\P(W_{k} \geq 0)}{k}s^{k} \right) =  \frac{1}{\Lambda \left(  \frac{1}{1-s} \right)}, \qquad s\in[0,1).\]
 Then, an application of Karamata's Tauberian theorem \cite[Theorem 8.1.6]{BGT89} (see also \cite[Theorem 3]{Rog71}) ensures that $\P(T_{1}>n)\sim \Lambda(n)$ as $n \rightarrow \infty$.
 \end{proof}
 
 \begin{remark} It is possible to relax the condition $\Pr{X_{1} < -1}=0$. Indeed, if $(X_{i} : i \geq 1)$ is a sequence of i.i.d.\ integer-valued random variables such that $\Es{X_{1}}=0$, $\Pr{X_{1} \geq x} \sim p {L(x)}/{x}$, $\Pr{X_{1} \leq -x} \sim q {L(x)}/{x}$ with $p+q=1$ (interpreted as $o(L(x)/x)$ if $p,q=0$). When $p>q$, the same proof using \cite[Lemma 7.3 and Eq.~(7.23)]{Ber17} shows the existence of a slowly varying function $\Lambda$ such that
 \[\P(\zeta>n)  \quad \mathop{\sim}_{n \rightarrow \infty} \quad   \frac{L(|{b}_{n}|)}{|b_{n}|} \cdot \Lambda(n), \qquad \P(T_{1}>n)  \quad \mathop{\sim}_{n \rightarrow \infty} \quad  \frac{1}{ \Lambda(n)}.\]
 \end{remark}

We state the following technical corollary in view of future use.

\begin{corollary}
\label{cor:utile}Let $(x_{n} : n \geq 1)$ be a sequence of positive real numbers such that $x_{n}=o(n)$ as $n \rightarrow \infty$. Then, the following estimates hold as $n \rightarrow \infty$:
\begin{enumerate}
\item  We have $\max_{1 \leq j \leq x_{n}} \left| \frac{\P(T_{1}>n-j)} {\P(T_{1}>n)}-1 \right| \rightarrow 0$.
\smallskip
\item We have $\frac{\P(\zeta>n)}{\P(\zeta \geq x_{n})} \rightarrow 0$.
\smallskip
\item We have $\Pr{X \geq |b_n|} \sim {\Pr{\zeta \geq n} \P(T_{1}>n)}$.
\end{enumerate}
\end{corollary}

\begin{proof}
For (i), first fix $n_{0}>1$ sufficiently large so that $ \lambda \coloneqq \inf_{m \geq n_{0}}(1-x_{m}/m)>0$. Then, for $n \geq n_{0}$, by monotonicity, for every $1 \leq j \leq x_{n}$, $ 1 \leq {\P(T_{1}>n-j)}/ {\P(T_{1}>n)} \leq  \P(T_{1}>\lambda n)/\P(T_{1}>n)$. By \cref{prop:laddertimes} $(i)$, the last quantity tends to $1$ as $n \rightarrow \infty$, which yields the first assertion.

For the second assertion, observe that by \cref{prop:laddertimes} $(i)$ one may write $\P(\zeta \geq n)= {\ell(n)}/{n}$ with $\ell$ a slowly varying function.  Again by Potter's bound, there exists a constant $A>0$ such that for every $n \geq 1$, $\ell(n)/\ell(x_{n}) \leq A (n/x_{n})^{1/2}$, so that 
\[\frac{\P(\zeta>n)}{\P(\zeta \geq x_{n})} \leq A \left(  \frac{x_{n}}{n} \right)^{1/2}  \quad \xrightarrow[n\to\infty]{} \quad 0.\]

The last assertion is an immediate consequence of \cref{prop:laddertimes} $(i)$, since $\P(X \geq |b_{n}|)= \frac{L(|b_{n}|)}{|b_{n}|}$. This completes the proof.
\end{proof}

The following estimate concerning the asymptotic behavior of $I_{n}$, the last weak ladder time of $(W_{i}: 0 \leq i \leq n)$, will be important.

\begin{proposition}
\label{prop:In}
Let $\tilde{\Lambda}$ be any continuous increasing slowly varying function such that, as $n \rightarrow \infty$,  $ \P(T_{1}>n)\sim \frac{1}{ \tilde\Lambda(n)}$. The following assertions hold as $n \rightarrow \infty$:
\begin{enumerate}
\item For every $x\in(0,1)$,
$\P \left( \frac{{\tilde\Lambda}(I_{n})}{{\tilde\Lambda}(n)} \leq x \right)  \rightarrow  x$.
\smallskip
\item  The convergence $ \frac{I_{n}}{n}  \rightarrow  0$ holds in probability. 
\end{enumerate}
\end{proposition}
Note that one could for instance take $\tilde\Lambda$ to be the function $\Lambda$ provided by \cref{prop:laddertimes}. 
\begin{proof}[Proof of \cref{prop:In}]
Let $(x_{n} : n\geq 1)$ be a sequence of positive real numbers such that $x_{n} \rightarrow \infty$ and $x_{n}=o(n)$ as $n \rightarrow \infty$. Let us first show that
\begin{equation}
\label{eq:In1} \P(I_{n} \leq x_{n})  \quad \mathop{\sim}_{n \rightarrow \infty} \quad  \frac{\tilde\Lambda(x_{n})}{\tilde\Lambda(n)}.
\end{equation}
To establish this, we use Lemma \ref{lem:In} to write $\P(I_{n} \leq x_{n})=\sum_{j=0}^{x_{n}} \P(\zeta > j)\P(T_{1}>n-j)$. Then, by Corollary \ref{cor:utile} $(i)$ and \cref{prop:laddertimes}, we have $\Lambda(n) \sim \tilde{\Lambda}(n)$ as $n \rightarrow \infty$ and
\[ \P(I_{n} \leq x_{n})  \quad \mathop{\sim}_{n \rightarrow \infty} \quad  \P(T_{1}>n) \sum_{j=0}^{x_{n}} \P(\zeta > j)  \quad \mathop{\sim}_{n \rightarrow \infty} \quad  \P(T_{1}>n)  \tilde\Lambda(x_{n}) \quad \mathop{\sim}_{n \rightarrow \infty} \quad  \frac{\tilde\Lambda(x_{n})}{\tilde\Lambda(n)}.\]

Next, fix $x \in (0,1)$. Since $\tilde\Lambda$ is increasing and continuous, we may consider its inverse $\tilde\Lambda^{-1}$, so that
\[\P \left( \frac{{\tilde\Lambda}(I_{n})}{{\tilde\Lambda}(n)} \leq x \right)=\P \left( I_{n} \leq {\tilde\Lambda}^{-1}(x {\tilde\Lambda}(n)) \right).\]
We claim that $ {\tilde\Lambda}^{-1}(x {\tilde\Lambda}(n)) \rightarrow \infty$ and that $ {\tilde\Lambda}^{-1}(x {\tilde\Lambda}(n))=o(n)$ as $n\rightarrow \infty$. The first convergence is clear since $\tilde\Lambda \rightarrow \infty$. For the second one, argue by contradiction and assume that there is $\varepsilon>0$ such that along a subsequence ${\tilde\Lambda}^{-1}(x {\tilde\Lambda}(n)) \geq \varepsilon n$. Then $x {\tilde\Lambda}(n) \geq {\tilde\Lambda}(\varepsilon n)$ along this subsequence. But ${\tilde\Lambda}(\varepsilon n)/{\tilde\Lambda}(n) \rightarrow 1$ since ${\tilde\Lambda}$ varies slowly. This implies $x \geq 1$, a contradiction. These claims then allow to use \eqref{eq:In1}:
\[\P \left( \frac{{\tilde\Lambda}(I_{n})}{{\tilde\Lambda}(n)} \leq x \right)  \quad \mathop{\sim}_{n \rightarrow \infty} \quad  \frac{\tilde\Lambda( {\tilde\Lambda}^{-1}(x {\tilde\Lambda}(n))))}{\tilde\Lambda(n)}= \frac{x \tilde\Lambda(n)}{\tilde\Lambda(n)}=x,\]
which establishes $(i)$.

For the assertion $(ii)$, fix $\varepsilon>0$. Then the previous paragraph shows that for every fixed $x \in (0,1)$, we have $x {\tilde\Lambda}(n) \leq {\tilde\Lambda}(\varepsilon n)$ for $n$ sufficiently large. Hence
\[\P(I_{n} \leq \varepsilon n)=\P(\tilde\Lambda(I_{n}) \leq \tilde\Lambda(\varepsilon n)) \geq  \P(\tilde\Lambda(I_{n}) \leq x \tilde\Lambda(n) ) \quad \xrightarrow[n\to\infty]{} \quad x.\]
Since this is true for every $x \in (0,1)$, it follows that $\P(I_{n} \leq \varepsilon n) \rightarrow 1$.
Observe that one could also obtain this as a consequence of the functional convergence \eqref{eq:cvfonctionnelle}, thanks to the definition of~$I_{n}$. 
\end{proof}

We conclude this section with an estimate concerning the number of weak ladder times up to time $n$.   Recall that $(T_{i} : i \geq 0)$ denotes the sequence of (weak) ladder times of $(W_{i} : i \geq 0)$. 
   
\begin{proposition}
\label{prop:numberladdertimes}
Let $H_{n}= \# \{ i \geq 0 : T_{i} \leq n\}$ be the number of weak ladder times of $(W_{i}: 0 \leq i \leq n)$.
We have 
 \[ \P(T_{1}>n) \cdot {H_{n}}   \quad \xrightarrow[n\to\infty]{(d)} \quad \mathsf{Exp}(1).\]
 \end{proposition}
 \begin{proof}
Fix $x>0$. Writing $\P(T_{1}>n)=r(n)$ to simplify notation, by \cite[Theorem 4.1]{Dar52},
\[\P \left( n r(T_{n}) \geq x \right)  \quad \xrightarrow[n\to\infty]{} \quad e^{-x}.\]
By replacing $n$ with $\lceil x/ r(n) \rceil $ we get that 
\[\P \left(r(T_{\lceil x /r(n) \rceil}) \geq r(n)\right) \quad \xrightarrow[n\to\infty]{} \quad e^{-x}.\]
But, since $r$ is decreasing,  we have
\[\P \left( H_{n} >  \frac{x}{r(n)} \right)=\P \left( T_{\lceil x/ r(n) \rceil } \leq n\right)=\P \left(r(T_{\lceil x /r(n) \rceil}) \geq r(n)\right)   \quad \xrightarrow[n\to\infty]{} \quad e^{-x}.\]
This completes the proof.
 \end{proof}

\subsection{Improvement in the local setting}\label{sec:EstimatesLocal}

In the previous section, we have established estimates on the random walk associated with the {\L}ukasiewicz path of a $\BGW_\mu$ tree under the assumption \eqref{eq:hypmu}. The goal of this section is to discuss one improvement in these estimates under the stronger assumption \eqref{eq:hypmustar}. In terms of the {\L}ukasiewicz path, this translates into the following assumption on the increment $X$.

\begin{center}
\uwave{\hspace{\linewidth}}
\begin{equation}
   \label{eq:hypXstar}
  \Es{X}=0, \qquad \Pr{X = x}\quad \mathop{\sim}_{x \rightarrow \infty} \quad\frac{L(x)}{x^2}  \qquad \text{and} \qquad  \Pr{X < -1}=0 .
  \tag{$\textrm{H}^{\mathrm{loc}}_X$}
   \end{equation}
\uwave{\hspace{\linewidth}}
\end{center} 
\medskip

Note that under these assumptions $\P(X_{1} \geq x) \sim \frac{L(x)}{x}$ as $x \rightarrow \infty$ so that \eqref{eq:hypX} is satisfied and the results of \cref{sec:EstimatesGeneralSetting} also hold. In this new setting, our main input is the following estimate, due to Berger \cite[Theorem 2.4]{Ber17}:
\begin{equation}
\label{eq:equivlocal}\P(W_{n}=-1)  \quad \mathop{\sim}_{n \rightarrow \infty} \quad n \frac{L(|b_{n}|)}{|b_{n}|^{2}}.
\end{equation}
(Indeed, we apply \cite[Theorem 2.4 $(i)$]{Ber17} with $x=-\lfloor b_{n} \rfloor-1$.)

Recall that $\ell^{*}$ is the slowly varying function defined by $\ell^{*}(n)=\sum_{k= n}^{\infty} \frac{L(k)}{k}$, which satisfies $L(n)/\ell^{*}(n) \rightarrow 0$ as $n \rightarrow \infty$. The next result identifies the slowly varying function $\Lambda$ in \cref{prop:laddertimes} under the assumptions of this section.

\begin{lemma}
\label{lem:estimateslocal}
The following estimates hold as $ n \rightarrow \infty$.
\begin{enumerate}
\item $\P(\zeta=n) \sim  \frac{L(|b_{n}|)}{b_{n}^{2}}$ and  $\P(\zeta \geq n)  \sim  \frac{n L(|b_{n}|)}{b_{n}^{2}}$.
\smallskip
\item $\P(T_{1}>n) \sim \ell^{*}(|b_{n}|) \sim \ell^{*}(a_{n})$.
\end{enumerate}
\end{lemma}

\begin{proof}
The first estimate readily follows from \eqref{eq:equivlocal}, since $\P(\zeta=n)= \frac{1}{n} \P(W_{n}=-1)$ by Kemperman's formula (see e.g.~\cite[Section~6.1]{Pit06}). 

For $\P(\zeta \geq n)$, note that $|b_{n}| \sim  n \ell^{*}(|b_{n}|) \sim n \ell^{*}(a_{n})$  (by \cite[ Lemma 7.3 \& Lemma 4.3]{Ber17}), so we have $\P(\zeta=n) \sim \frac{1}{n^{2}} \cdot \frac{L(|b_{n}|)}{\ell^{*}(|b_{n}|)^{2}}$. Since moreover $\frac{L(|b_{n}|)}{\ell^{*}(|b_{n}|)^{2}}$ is slowly varying, we get by \cite[Proposition 1.5.8]{BGT89} that
\[\P(\zeta \geq n)  \quad \mathop{\sim}_{n \rightarrow \infty} \quad  \frac{1}{n}\cdot \frac{L(|b_{n}|)}{\ell^{*}(|b_{n}|)^{2}} \sim  \frac{n L(|b_{n}|)}{b_{n}^{2}}.\]

The second assertion follows from the first and the fact that $\P(T_{1}>n) \sim \frac{\P(X \geq |b_{n}|)}{\P(\zeta \geq n)}$ by \cref{cor:utile} $(iii)$.\end{proof}

\begin{remark}
\label{rem:constants}
In the specific case where $\P(X=n) \sim \frac{L(n)}{n^{2}}$ as $n \rightarrow \infty$, one may thus take $\Lambda(n)= \frac{1}{\ell^{*}(a_{n})}$ or  $\Lambda(n)= \frac{1}{\ell^{*}(|b_{n}|)}$.
\end{remark}

The following example follows from \cref{lem:estimateslocal} and may help to visualize the different orders of magnitude.

\begin{example}\label{ex:1} Assume that $\mu(n) \sim \frac{c}{n^{2} \ln(n)^{2}}$ as $n \rightarrow \infty$. Then, as $n\rightarrow \infty$,
\begin{equation*}
a_{n} \sim \frac{cn}{\ln(n)^{2}}, \quad b_{n} \sim - \frac{cn}{\ln(n)}, \quad \P(|X| \geq b_{n}) \sim \frac{1}{n \ln(n)}, \quad \P(\zeta \geq n) \sim \frac{1}{c^{2} n}, \quad \P(T_{1}>n) \sim \frac{c^{2}}{\ln(n)}.
\end{equation*}
\end{example}

\section{Cauchy random walks: local conditioning}
\label{sec:local}

The ultimate goal of this section is to study a $\BGW_\mu$ tree $\Tn$ conditioned to have $n$ vertices, when the offspring distribution $\mu$ satisfies \eqref{eq:hypmustar}.

\smallskip

To this end, we consider a random walk $(W_i : i\geq 0)$ whose increments satisfy assumption \eqref{eq:hypXstar}. We aim at studying the behaviour of the \textit{excursion} $(W^{(n)}_{i} : i \geq 0)$, whose law is that of the random walk $(W_{i} : i \geq 0)$ under the conditional probability $ \P( \, \cdot \, | \zeta = n)$, and which is also the {\L}ukasiewicz path of the random tree $\Tn$. (Note that our assumptions imply that $\P(\zeta=n)>0$ for every $n$ sufficiently large, see \cref{lem:estimateslocal}). 

More precisely, we shall couple with high probability the trajectory $(W^{(n)}_{i} : i \geq 0)$ with that of a random walk conditioned to be nonnegative for a random number of steps (whose number converges in probability to $\infty$ as $n \rightarrow \infty$), followed by an independent ``big jump'', and then followed by an independent unconditioned random walk.  This allows us to obtain a functional invariance principle for $\Wn$ which is of independent interest (\cref{thm:RW-local}).

We will use the notation and results of Section \ref{sec:Estimates}.

\subsection{Bridge conditioning}

In order to study the excursion $\Wn$, we start with some results on the \textit{bridge} that has the law of $(W_i : 0\leq i \leq n)$ under the probability measure $\P(\ \cdot \ | W_{n}=-1)$. Recall that $(X_i : i\geq 1)$ is a sequence of i.i.d.\ variables distributed as $X$ and for every $n\in\N$, let 
\[V_n \coloneqq \inf \left\{1\leq j \leq n : X_j=\max\{X_i : 1\leq i \leq n\} \right\}\]
be the first index of the maximal element of $(X_{1}, \ldots,X_{n})$. Then, we denote by $(X^{(n)}_{1}, \ldots,X^{(n)}_{n-1})$ a random variable distributed as $(X_{1},\ldots,X_{V_{n}-1},X_{V_{n}+1} \ldots,X_{n})$ under $\P(\ \cdot \ | W_{n}=-1)$.

\begin{proposition}\label{prop:dTV1}
We have
\[d_{\mathrm{TV}}\left( \left(\Xn_{i} : 1 \leq i  \leq n-1 \right),  \left(X_{i} : 1 \leq i \leq n-1 \right) \right)  \quad \xrightarrow[n\to\infty]{}  \quad 0,\]
where $d_{\mathrm{TV}}$ denotes the total variation distance on $\R^{n-1}$ equipped with the product topology.
\end{proposition}

We refer to \cite{Lin92} or \cite[Section~2]{dHo12} for background concerning the total variation distance. The proof is inspired from that of \cite[Theorem 1]{AL11}. Since the context is different, we give a detailed proof.

\begin{proof}  For every $A \in \mathcal{B}(\R^{n-1})$, note that  $\P( (X_{1},\ldots,X_{V_{n}-1},X_{V_{n}+1} \ldots,X_{n}) \in A, W_{n}=-1)$ is bounded for $n$ sufficiently large from below by the probability of the event
\begin{align*}
& \bigcup_{i=1}^{n} \biggl\{ \left(X_{1}, \ldots, X_{i-1}, X_{i+1}, \ldots,X_{n-1}\right) \in A,    \Bigl|\sum_{{1 \leq j \leq n , j \neq i}} X_{j}+|b_{n}| \Bigr|\leq K a_{n},\\
 &  \qquad  \qquad  \qquad  \qquad  \qquad \max_{{1 \leq j \leq n , j \neq i}} X_{j} < |b_{n}|-1-K a_{n},W_n=-1 \biggr\},
\end{align*}
 where $K>0$ is an arbitrary constant and the events appearing in the union are disjoint
By cyclic invariance of the law of $(X_{1}, \ldots,X_{n})$, we get that $\P( (X_{1},\ldots,X_{V_{n}-1},X_{V_{n}+1} \ldots,X_{n}) \in A, W_{n}=-1)$  is bounded from below by
\begin{align*}
 n\P \biggl((X_{1}, \ldots,X_{n-1}) \in A, &\Bigl| W_{n-1}+|b_{n}|\Bigr| \leq K a_{n},\\
 & \max_{1 \leq j \leq n-1} X_{j}<|b_{n}|-1-K a_{n}, W_{n}=-1\biggr).
 \end{align*} Let us introduce the event
\[G_{n}(K)\coloneqq \left\{ \Bigl| W_{n-1}+|b_{n}|\Bigr| \leq K a_{n}, \max_{1 \leq j \leq n-1} X_{j}<|b_{n}|-1-K a_{n} \right\}.\] Since $\P(X=n)$ is regularly varying, observe that 
\[\P(X= -1-k_{n})  \quad \mathop{\sim}_{n \rightarrow \infty} \quad   \P(X=|b_{n}|)\]
uniformly in $k_{n}$ satisfying $|k_{n}-|b_{n}|| \leq K a_{n}$ (because $a_{n}/|b_{n}| \rightarrow 0$). Moreover, by \eqref{eq:equivlocal} we have that $\P(W_{n}=-1) \sim n\P(X=|b_{n}|)$. Therefore, there exists a sequence $\varepsilon_{n} \rightarrow 0$ such that
\begin{align*}
&\P\left( \left(X^{(n)}_{1}, \ldots,X^{(n)}_{n-1}\right) \in A \right)\\
& \qquad  \qquad  \qquad  \qquad  \geq  (1-\varepsilon_{n}) \P \left((X_{1}, \ldots,X_{n-1}) \in A, G_{n}(K)\right)\\
& \qquad  \qquad  \qquad  \qquad \geq  (1-\varepsilon_{n}) \left(\P ((X_{1}, \ldots,X_{n-1}) \in A) -\P\left(\overline{G_{n}(K)}\right)\right).
\end{align*}
Hence
\begin{align*}
&\P\left(\left(X^{(n)}_{1}, \ldots,X^{(n)}_{n-1}\right) \in A\right)- \P((X_{1}, \ldots,X_{n-1}) \in A) \\
& \qquad   \qquad \geq - \varepsilon_{n}  \P ((X_{1}, \ldots,X_{n-1}) \in A)  -(1-\varepsilon_{n}) \P\left(\overline{G_{n}(K)}\right).
\end{align*}
By writing this inequality with $\overline{A}$ instead of $A$, we get that 
\begin{align*}
& \left|\P\left(\left(X^{(n)}_{1}, \ldots,X^{(n)}_{n-1}\right) \in A \right)- \P((X_{1}, \ldots,X_{n-1}) \in A) \right| \\
&\qquad  \qquad  \qquad  \qquad  \qquad \qquad \qquad \qquad  \leq \varepsilon_{n}+ \P\left(\overline{G_{n}(K)}\right).
\end{align*}
It therefore remains to check that 
\[\limsup_{K \rightarrow \infty} \limsup_{n \rightarrow \infty}  \P\left(\overline{G_{n}(K)}\right)=0.\]
To this end, first notice that by \eqref{eq:cvfonctionnelle},
\[ \P\left( \Bigl| W_{n-1}+|b_{n}|\Bigr| > K a_{n} \right)  \quad \xrightarrow[n\to\infty]{} \quad \P( \mathcal{C}_{1}>K), \]
where $ \mathcal{C}$ is an asymmetric Cauchy process. Since $  \mathcal{C}_{1}$ is almost surely finite, we have $\P( \mathcal{C}_{1}>K) \rightarrow 0$ as $K \rightarrow \infty$. Second, write
\[\P \left( \max_{1 \leq j \leq n-1} X_{j} \geq |b_{n}|-1-K a_{n}\right)=1- \left( 1-\P(X \geq  |b_{n}|-1-K a_{n})) \right)^{n-1}.\]
But
\[(n-1) \P(X \geq  |b_{n}|-1-K a_{n}) \sim  \frac{nL( |b_{n}|)}{ |b_{n}|}  \sim \frac{L(|b_{n}|)}{\ell^{*}(|b_{n}|)} \quad \xrightarrow[n\to\infty]{} \quad 0.\]
This completes the proof.
\end{proof}

\subsection{Excursion conditioning}
\label{sec:excursioncond}

We now deduce from Proposition \ref{prop:dTV1} a result on the excursion $W^{(n)}$. To do so, we will use the so-called Vervaat transform $\mathcal{V}$ which is defined as follows. Let $n\in\N$, $(x_1,\ldots, x_n) \in \Z^n$ and let $\mathbf{w}=(w_i : 0 \leq i \leq n)$ be the associated walk defined by
\[w_0=0 \quad \text{and} \quad w_i=\sum_{j=1}^i x_j, \quad 1\leq i \leq n.\] We also introduce the first time at which $(w_i : 0\leq i \leq n)$ reaches its overall minimum,
\[k_n:=\min\{0\leq i \leq n : w_i=\min\{w_j : 0\leq j \leq n\}\}.\]
The Vervaat transform $\mathcal{V}(\mathbf{w}) \coloneqq (\mathcal{V}(\mathbf{w})_i : 0\leq i \leq n)$ of $\mathbf{w}$ is the walk obtained by reading the increments $(x_1,\ldots, x_n)$ from left to right in cyclic order, started from $k_n$. Namely,  \[\mathcal{V}(\mathbf{w})_0=0 \quad \text{and} \quad \mathcal{V}(\mathbf{w})_{i+1}-\mathcal{V}(\mathbf{w})_{i}= x_{k_n+i \mod[n]}, \quad 0\leq i < n,\] see Figure \ref{fig:Vervaat} for an illustration. 

Recall that $(W_i : i \geq 0)$ is a random walk with increments distributed as $X$, and for every $n\in\N$ define the random process $\Zn \coloneqq (\Zn_{i} : 0\leq i  \leq n )$  by 
\begin{equation}\label{eqn:ZnVervaat}
	\Zn \coloneqq \mathcal{V}(W_0,W_1,\ldots,W_{n-1},-1).
\end{equation} 

The next result shows that $(W^{(n)}_{i} : 0\leq i  \leq n)$ and $(\Zn_{i} : 0\leq i  \leq n )$ are close in the total variation sense when $n$ goes to infinity, where we recall that $(W^{(n)}_{i} : 0\leq i  \leq n)$ is the random walk $(W_{i} : i \geq 0)$ under the conditional probability   $ \P( \, \cdot \, | \zeta = n)$.

\begin{theorem}\label{thm:dTVLoc}
We have
\[d_{\mathrm{TV}}\left( \left(W^{(n)}_{i} : 0\leq i  \leq n \right),  \left(\Zn_{i} : 0\leq i  \leq n \right) \right)  \quad \xrightarrow[n\to\infty]{}  \quad 0,\]
where $d_{\mathrm{TV}}$ denotes the total variation distance on $\R^{n+1}$ equipped with the product topology.
\end{theorem}

\begin{proof} Throughout the proof, we let $\Bn \coloneqq (\Bn_{i} : 0\leq i  \leq n)$ be a bridge of length $n$, that is, a process distributed as $(W_{i} : 0\leq i  \leq n)$ under $\P(\ \cdot \ | W_{n}=-1)$.  For every $0\leq i < n$, we denote by $b^{(n)}_i := \Bn_{i+1}-\Bn_{i}$ the $i$-th increment of the bridge. We will need the first time at which $(\Bn_{i} : 0\leq i  \leq n)$ reaches its largest jump, defined by
\[V_n^b \coloneqq \inf\left\{0\leq i < n : b^{(n)}_i=\max\left\{b^{(n)}_j : 0\leq j < n\right\} \right\}.\] Without loss of generality, we assume that the largest jump of $B^{(n)}$ is reached once. We finally introduce the shifted bridge $\Rn\coloneqq (\Rn_{i} : 0\leq i  \leq n)$, obtained by reading the jumps of the bridge $B^{(n)}$ from left to right starting from $V^b_n$. Namely, we set \[\Rn_0=0 \quad \text{and} \quad r^{(n)}_i \coloneqq \Rn_{i+1}-\Rn_{i}= b^{(n)}_{V_n^b+i+1 \mod[n]}, \quad 0\leq i < n,\] see  Figure \ref{fig:Vervaat} for an illustration. 

\begin{figure}[h!]
	\centering
	\includegraphics[width=\linewidth]{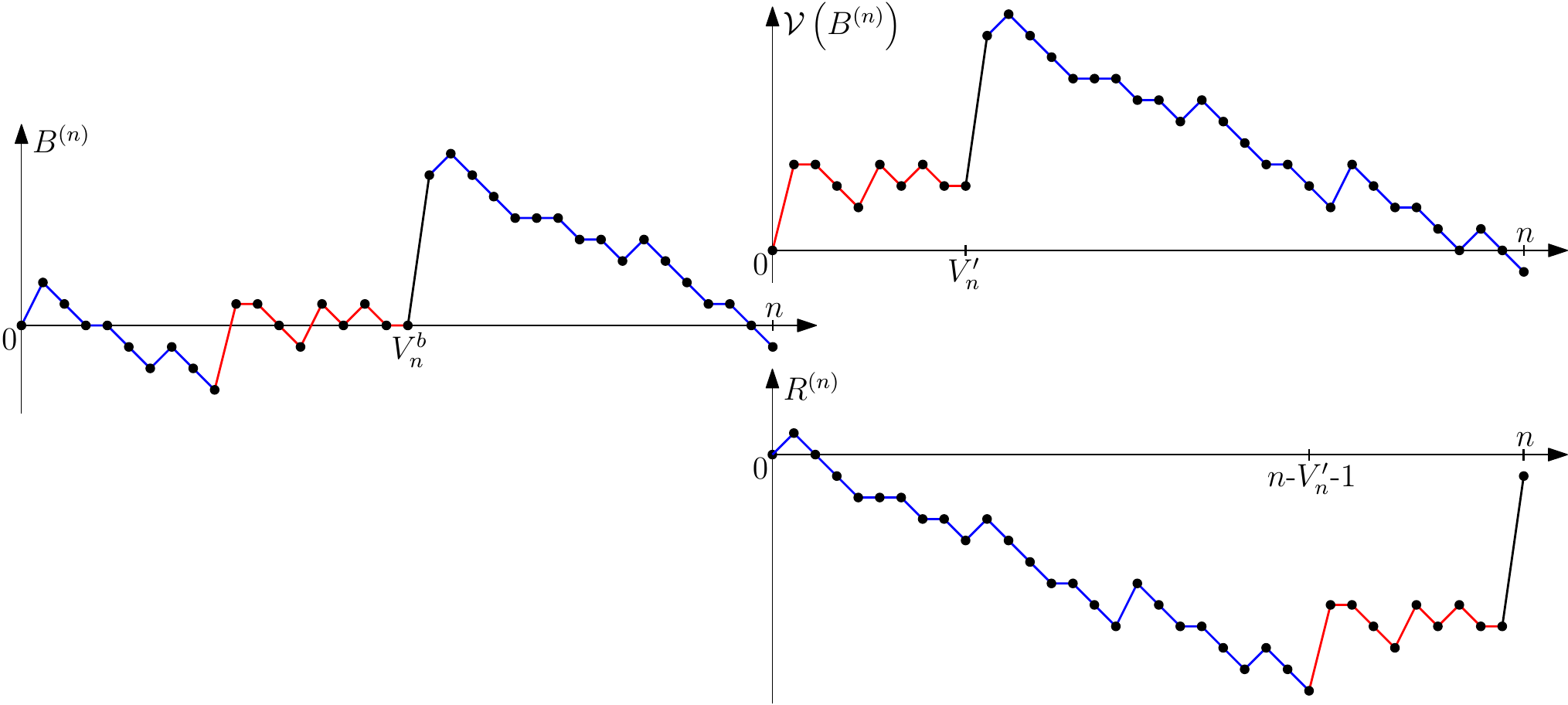}
	\caption{The bridge $B^{(n)}=(B^{(n)}_{i} : 0\leq i  \leq n)$ with the location $V_{n}^{b}$ of its (first) maximal jump, its Vervaat transform $ \mathcal{V}(\Bn) $  with the location $V_{n}'$ of its (first) maximal jump, and the shifted bridge $R^{(n)}=(R^{(n)}_{i} : 0\leq i  \leq n)$ with the location of its first overall minimum.}
	\label{fig:Vervaat}
\end{figure}

Since $V^b_n$ is independent of $(b^{(n)}_0,\ldots,b^{(n)}_{V_n^b-1},b^{(n)}_{V_n^b+1},\ldots,b^{(n)}_{n-1})$, we have 
\[\left( r^{(n)}_i : 0\leq i <n-1 \right)=\left( b^{(n)}_{V_n^b+i+1 \mod[n]} : 0\leq i <n-1 \right)\overset{(d)}{=}\left(X^{(n)}_{1}, \ldots,X^{(n)}_{n-1}\right).\] 
One can then apply Proposition \ref{prop:dTV1} to get that 
\[d_{\mathrm{TV}}\left( \left(\Rn_{i} : 0\leq i <  n \right),  \left(W_{i} : 0\leq i  < n \right) \right)  \quad \xrightarrow[n\to\infty]{}  \quad 0.\] 

We now use the Vervaat transform. By construction, $\mathcal{V}(\Rn)= \mathcal{V}(\Bn) $ (see Figure \ref{fig:Vervaat}), and $ \mathcal{V}(\Bn) $  has the same distribution as the excursion $(W^{(n)}_{i} : 0\leq i  \leq n)$, see for instance \cite[Section 5]{Pit06}. Since $\Rn_n=-1$ and $\Zn=\mathcal{V}\left(W_0, \ldots,W_{n-1},-1\right)$ by definition, this concludes the proof. \end{proof}

Let us denote by $V^z_n$ the index of the first largest jump of $\Zn$, 
\[V_n^z \coloneqq \inf\left\{0\leq i < n : \Zn_{i+1}-\Zn_{i}=\max\left\{\Zn_{j+1}-\Zn_{j} : 0\leq j < n\right\} \right\}.\] Then, one can identify with high probability the law of $\Zn$ until time $V^z_{n}$ as follows. Let us denote by $(\widehat{W}_{i} :0 \leq i < n)$ the time reversed random walk defined by $\widehat{W}_{i}=W_{n-1}-W_{n-1-i}$ for $0 \leq i < n$, and let $\widehat{I}_{n-1}$  be the last weak ladder time of  $(\widehat{W}_{i} :0 \leq i < n)$. We have the following result.

\begin{corollary}
\label{cor:timereversal}
Let $\mathcal{E}_{n}$ be the event $\mathcal{E}_{n}\coloneqq \left\lbrace \max\left\lbrace X_i : 1\leq i < n \right\rbrace < -1-W_{n-1} \right\rbrace$. Then, 
\begin{enumerate}
\item[(i)]   We have $\P( \mathcal{E}_{n}) \rightarrow 1$ as $n \rightarrow \infty$;
\item[(ii)] On the event $ \mathcal{E}_{n}$, we have $(\Zn_{i} : 0\leq i  \leq V^z_{n})= (\widehat{W}_{i} : 0\leq i  \leq \widehat{I}_{n-1})$.
\end{enumerate}
\end{corollary}
\begin{proof} The event $\mathcal{E}_{n}$ can be rephrased as the fact that the maximal jump of the process $(W_0,W_1,\ldots,W_{n-1},-1)$ is the last one. First, observe that the function convergence \eqref{eq:cvfonctionnelle} combined with the continuity of the largest jump for the Skorokhod $J_{1}$ topology implies that $ \frac{1}{a_{n}} \max_{1 \leq j < n} X_{k}$ converges in distribution to a non-degenerate random variable. Moreover, $\tfrac{1}{a_{n-1}}(W_{n-1}-b_{n-1})$ converges in distribution so that $ \frac{1}{|b_{n}|}(-1-W_{n-1})$ converges in distribution to $1$. The first assertion then follows from the fact that $a_{n}=o(|b_{n}|)$, while the second is a simple consequence of the definition \eqref{eqn:ZnVervaat} of $\Zn$.\end{proof}

We now establish a functional invariance principle for $\Wn$. We set $W^{(n)}_{k}=0$ for $k<0$ by convention.

\begin{theorem}
\label{thm:RW-local}
The convergence
\[ \left(  \frac{\Wn_{\lfloor n t \rfloor}}{|b_{n}|} : -1 \leq t \leq 1 \right)  \quad \xrightarrow[n\to\infty]{(d)}  \quad   \left( (1- t) \mathbbm{1}_{t \geq 0} : -1 \leq t \leq 1\right)\] holds in distribution in $\D([-1,1],\R)$.
  \end{theorem}
  
Here, we work with $\D([-1,1],\R)$ instead of  $\D([0,1],\R)$ since our limiting process almost surely
takes a positive value in $0$ (it ``starts with a jump''),
while $\Wn$ stays small for a positive time (see Figure \ref{fig:luka} for a simulation).

\begin{proof}
 By \cref{thm:dTVLoc}, it is enough to establish the result with $\Wn$ replaced with $\Zn= \mathcal{V}(W_0,W_1,\ldots,W_{n-1},-1)$. Recall that $V^z_{n}$ is the index of the first largest jump of $\Zn$. Thanks to \cref{cor:timereversal}, we can also assume without loss of generality that $\mathcal{E}_n$ is realized, so that
\begin{enumerate}
\item[--]  $(\Zn_{i} : 0\leq i  \leq V^z_{n})=(\widehat{W}_{i} : 0\leq i  \leq \widehat{I}_{n-1})$;
\item[--] $\Zn_{V^z_{n}+1}-\Zn_{V^z_{n}}=-1-W_{n-1}$;
\item[--] $(\Zn_{V^z_{n}+1+i}-\Zn_{V^z_{n}+1} :  0 \leq i  < n-V^z_{n})=(W_{i}: 0 \leq i < n-\widehat{I}_{n-1})$.
\end{enumerate}

Since $(\widehat{W}_{i} : 0 \leq i < n)$ and $(W_{i}: 0 \leq i <n)$ have the same distribution, by \cref{prop:In} $(ii)$ and \eqref{eq:cvfonctionnelle}  we have the convergences
\[ \frac{V^z_{n}}{n}  \quad \xrightarrow[n\to\infty]{(\P)}\quad  0,   \qquad  \frac{1}{|b_{n}|}\max_{0 \leq i \leq V^z_{n}} \left|\Zn_{i}\right| \quad \xrightarrow[n\to\infty]{(\P)}\quad  0  \quad \text{and} \quad  \frac{1}{|b_{n}|}\left(\Zn_{V^z_{n}+1}-\Zn_{V^z_{n}}\right) \quad \xrightarrow[n\to\infty]{(\P)}\quad  1\]
as well as the convergence in distribution in $\D([0,1],\R)$
\[ \left(  \frac{1}{|b_{n}|} \left({\Zn_{V_{n}+1+\lfloor n t \rfloor}}-\Zn_{V_{n}+1}\right) :0 \leq t  \leq 1 \right) \quad \xrightarrow[n\to\infty]{(d)} \quad \left( - t  : 0 \leq t \leq 1\right),\]
where we set $Z^{(n)}_{k}=0$ for $k>n$. The desired result readily follows.\end{proof}

\subsection{Applications: limit theorems for BGW trees}\label{ss:AppliGWloc}

Throughout this section, we let $\mu$ be an offspring distribution satisfying \eqref{eq:hypmustar}, and let $ \Tn$ be a $ \BGW_{ \mu}$ tree conditioned on having $n$ vertices. We now apply the results of the previous sections to the study of the tree $\Tn$.

First of all, we immediately obtain a limit theorem for the {\L}ukasiewicz path $\W(\Tn)$ by simply combining \cref{prop:GWRW} with \cref{thm:RW-local}. Also note that \cref{thm:dTVLoc} gives a simple and efficient way to asymptotically simulate $\Tn$.

\begin{proposition}
  \label{prop:luka-local}
The convergence
\[ \left(  \frac{\W_{\lfloor n t \rfloor}(\Tn)}{|b_{n}|} : -1 \leq t \leq 1 \right)  \quad \xrightarrow[n\to\infty]{(d)}  \quad   \left( (1- t) \mathbbm{1}_{t \geq 0} : -1 \leq t \leq 1\right)\] holds in distribution in $\D([-1,1],\R)$.
  \end{proposition}

\begin{figure}[h!]
  \centering
  \includegraphics[width=0.5 \linewidth]{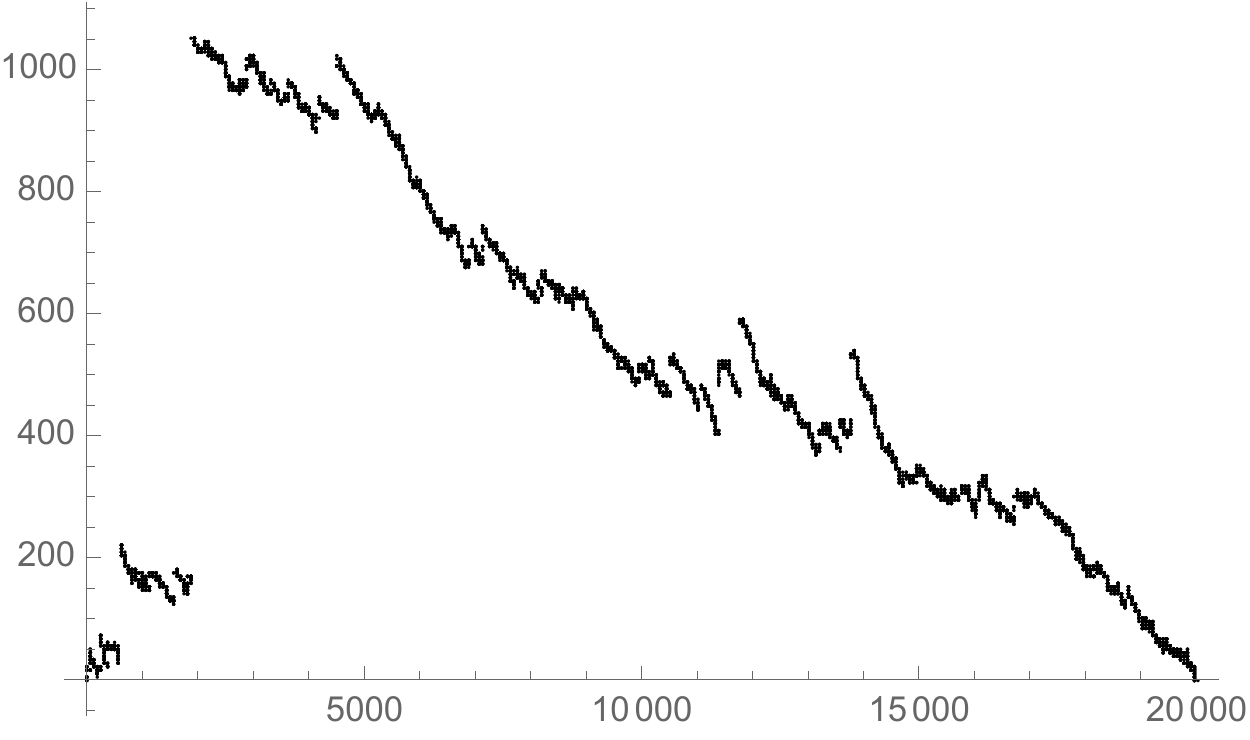}
\caption{\label{fig:luka} The {\L}ukasiwiecz path of the tree depicted in  Figure \ref{fig:tree}.}
\end{figure}

Our goal is now to prove Theorem \ref{thm:loop-local}, which requires more work.
 
\begin{proof}[Proof of Theorem \ref{thm:loop-local}] First of all, by \cref{prop:GWRW} and  \cref{thm:dTVLoc}, we can work with the tree $\Tn'$ whose {\L}ukasiewicz path is $\Zn$ instead of $\Tn$. Recall that by \eqref{eqn:ZnVervaat}, $\Zn \coloneqq \mathcal{V}(W_0,W_1,\ldots,W_{n-1},-1)$. 

By \cref{cor:timereversal} $(i)$, we have $\Delta(\Tn')= \vert W_{n-1} \vert +1$ with probability tending to one as $n\rightarrow \infty$, which yields the first part of the statement, that is,
\[\frac{\Delta( \mathcal{T}'_{  n})-|b_{n}| }{a_{n}}   \quad \xrightarrow[n\to\infty]{(d)} \quad \mathcal{C}_{1}.\]

We now turn to the second part. We denote by $v'_n$ the vertex of maximal degree in $\Tn'$, and we work without loss of generality conditionally on the event that this vertex is unique and has degree $\Delta(\Tn')= \vert W_{n-1} \vert +1$.
We let $(\tau_k : 0\leq k \leq \vert W_{n-1} \vert)$ be the connected components of $\Tn'\backslash \{v'_n\}$ (cyclically ordered), where $\tau_0$ is the connected component containing the root vertex of $\Tc'_n$. For every $k\neq 0$, we assume that $v'_n$ is the root vertex of $\tau_k$. By construction, $\Loop(\Tn')$ can be described as a cycle of length $\Delta(\Tn')$ on which the random graphs $(\Loop(\tau_k) : 0 \leq k \leq \vert W_{n-1} \vert)$ are grafted. Our goal is to prove the following estimate:
 \begin{equation}\label{eqn:DiamLoops}
 	\frac{1}{\vert b_n \vert}\sup_{0\leq k \leq \vert W_{n-1} \vert} \rad\left(\Loop\left( \tau_k\right)\right) \quad \xrightarrow[n \rightarrow \infty]{(\P)} \quad 0,
 \end{equation} where $\rad(G)$ stands for the radius of the pointed graph $G$, that is, the maximal graph distance to the root vertex in $G$ (implicitly, $\tau$ and $\Loop(\tau)$ share the same root vertex). Then, the desired result will follow from standard properties of the Gromov--Hausdorff topology.

 Let us introduce a decomposition of the walk $(W_i : i \geq 0)$ into excursions above its infimum. Namely, we set $\zeta_{k}=\inf\{i\geq 0 : W_i=-k\}$  for every $k\geq 0$, and introduce the excursions
 \[\left( \mathcal{W}^{(k)}_i : 0 \leq i \leq \zeta_{k}-\zeta_{k-1} \right) \coloneqq \left( W_{\zeta_k+i}+k : 0 \leq i \leq \zeta_{k}-\zeta_{k-1} \right), \quad k\in \N.\] For every $k\geq 1$, we let $\tau_k$ be the tree whose {\L}ukasiewicz path is $\mathcal{W}^{(k)}$. This choice of notation is justified by the fact that for every $1\leq k \leq \vert W_{n-1}\vert$, $\tau_k$ is indeed the $k$-th tree grafted on $v'_n$ in $\Tn'$ (see Figure \ref{fig:DecompoLoop} for an illustration).
 
 The ancestral tree $\tau_0$ plays a special role. If we set $\underline{W}_k:=\inf\{W_i : 0 \leq i \leq k\}$ for every $k \geq 1$, then its {\L}ukasiewicz path is given by
 \begin{align*}
 	\Bigg[\left( W_{\zeta_{\vert\underline{W}_{n-1}\vert}+i}-\underline{W}_{n-1}\right)_{0\leq i < n-\zeta_{\vert\underline{W}_{n-1}\vert}},& W_{n-1}-\underline{W}_{n-1}-1 , \\
 		 &\left( W_{\zeta_{\vert W_{n-1}\vert}+i}-\underline{W}_{n-1}-1\right)_{ 0\leq i < \zeta_{\vert\underline{W}_{n-1}\vert}-\zeta_{\vert W_{n-1}}\vert}\Bigg].
 \end{align*} Thus, we can decompose this tree into:
\begin{itemize}
	\item The tree $\tau^*_{\vert\underline{W}_{n-1}\vert+1}$ whose {\L}ukasiewicz path is $(W_{\zeta_{\vert\underline{W}_{n-1}\vert}+i}-\underline{W}_{n-1} : 0\leq i < n-\zeta_{\vert\underline{W}_{n-1}\vert})$ (completed by $-1$ steps). This is the tree made of the spine $\llbracket \varnothing, v'_n \llbracket$ together with children of its vertices and all descendants on its left in $\Tn'$.
	\item The trees $(\tau_k : \vert W_{n-1}\vert < k \leq \vert \underline{W}_{n-1} \vert)$ that are the $W_{n-1}-\underline{W}_{n-1}$ trees grafted on the right of the spine $\llbracket \varnothing, v'_n \llbracket$ in $\Tn'$.
\end{itemize} This decomposition is illustrated in Figure \ref{fig:DecompoLoop}.

\begin{figure}[h!]
  \centering
  \includegraphics[width=\linewidth]{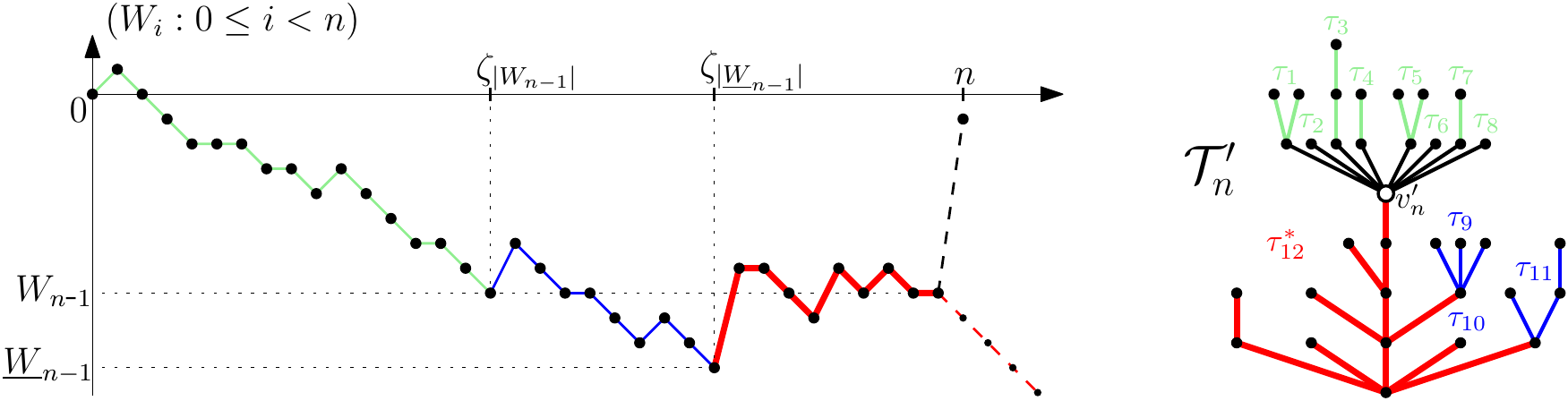}
\caption{The random walk $(W_i : 0\leq i < n)$ and the associated tree $\Tn'$. In light green, the $\vert W_{n-1} \vert $ first excursions of $W$ that encode the trees grafted above the vertex with maximal degree $v'_n$. In blue, the next $W_{n-1}-\underline{W}_{n-1}$ excursions of $W$, that encode the trees $(\tau_k : \vert W_{n-1}\vert < k \leq \vert \underline{W}_{n-1} \vert)$. In bold red, the tree $\tau^*_{\vert\underline{W}_{n-1}\vert+1}$ and its {\L}ukasiewicz path.}
\label{fig:DecompoLoop}
\end{figure}

Back to \eqref{eqn:DiamLoops}, we get that
\begin{align*}
	\sup_{0\leq k < \Delta(\Tn') } \rad\left(\Loop\left( \tau_k\right)\right)& \leq \rad\left(\Loop\left( \tau_0\right)\right)+\sup_{1\leq k \leq \vert W_{n-1} \vert} \rad\left(\Loop\left( \tau_k\right)\right)\\
&\leq \rad\left(\Loop\left( \tau^*_{\vert\underline{W}_{n-1}\vert+1}\right)\right)+ \sup_{1\leq k \leq \vert \underline{W}_{n-1} \vert}\rad\left(\Loop\left( \tau_k\right)\right).
\end{align*} By standard estimates on looptrees (see \cite[Lemma 11]{KR18}), for every plane tree $\tau$ we have 
 \begin{equation*} 	\rad\left(\Loop\left( \tau\right)\right)\leq \mathbf{H}(\tau) + \sup_{0 \leq i < \vert \tau \vert} \W_i(\tau),
 \end{equation*} where $\mathbf{H}(\tau)$ is the  height (i.e.\ the radius) of $\tau$ and $\W(\tau)$ its {\L}ukasiewicz path. This yields
\begin{align*}
	\sup_{0\leq k < \Delta(\Tn') } \rad\left(\Loop\left( \tau_k\right)\right) \leq \mathbf{H}\left( \tau^*_{\vert\underline{W}_{n-1}\vert+1}\right)+ \sup_{1\leq k \leq \vert \underline{W}_{n-1} \vert}\mathbf{H}\left( \tau_k\right) +2\sup_{0 \leq i \leq n} \left(W_i-\underline{W}_i \right).
\end{align*} By the functional convergence \eqref{eq:cvfonctionnelle} and since $a_{n}=o(|b_{n}|)$, we have
\[\frac{1}{\vert b_n \vert}\sup_{0 \leq i \leq n} \left(W_i-\underline{W}_i \right) \quad \xrightarrow[n \rightarrow \infty]{(\P)} \quad 0,\] so that it suffices to show that
\[\frac{1}{\vert b_n \vert} \left( \mathbf{H}\left( \tau^*_{\vert\underline{W}_{n-1}\vert+1}\right)+ \sup_{1\leq k \leq \vert \underline{W}_{n-1} \vert}\mathbf{H}\left( \tau_k\right) \right) \quad \xrightarrow[n \rightarrow \infty]{(\P)} \quad 0.\] Since $\tau^*_{\vert\underline{W}_{n-1}\vert+1}$ is a subtree of $\tau_{\vert\underline{W}_{n-1}\vert+1}$ (that is, the tree encoded by the $(\vert\underline{W}_{n-1}\vert+1)$-th excursion of $W$), we obtain
\[\mathbf{H}\left( \tau^*_{\vert\underline{W}_{n-1}\vert+1}\right)+ \sup_{1\leq k \leq \vert \underline{W}_{n-1} \vert}\mathbf{H}\left( \tau_k\right) \leq \sup_{1\leq k \leq \vert \underline{W}_{n-1} \vert +1}\mathbf{H}\left( \tau_k\right).\] Moreover, we have
\[ \Pr{\sup_{1\leq k \leq \vert \underline{W}_{n-1} \vert +1}\mathbf{H}\left( \tau_k\right) \geq \varepsilon \vert b_n \vert }\leq \Pr{\vert\underline{W}_{n-1}\vert \geq 2 \vert b_n \vert} +  \Pr{\sup_{1\leq k \leq 2 \vert b_n \vert }\mathbf{H}\left( \tau_k\right) \geq \varepsilon \vert b_n \vert }.\] Thanks to the functional convergence \eqref{eq:cvfonctionnelle}, we have $\Pr{\vert\underline{W}_{n-1}\vert \geq 2 \vert b_n \vert} \rightarrow 0$ as $n \rightarrow \infty$. Then, recall that $(W_i : i \geq 0)$ is a random walk, so that the trees $(\tau_k : k\geq 1)$ are i.i.d. $\BGW_\mu$ trees. It follows from Lemma \ref{lem:HeightBGW} that 
\[\Pr{\sup_{1\leq k \leq 2 \vert b_n \vert }\mathbf{H}\left( \tau_k\right) \geq \varepsilon \vert b_n \vert } \sim 2 \vert b_n \vert\Pr{\mathbf{H}\left( \tau_k\right) \geq \varepsilon \vert b_n \vert} \quad \xrightarrow[n \rightarrow \infty]{} \quad 0.\] This proves \eqref{eqn:DiamLoops} and completes the proof.
\end{proof}

We now  obtain information concerning the vertex with maximal degree in $\Tn$, such as its height and its index in the lexicographical order. In this direction,  in virtue of \cref{prop:laddertimes}, we let $\Lambda$ be the increasing slowly varying function such that
\[ \P(\inf \{j > 0: W_{j} \geq 0\} \geq n)=\frac{1}{\Lambda(n)}.\] We also denote by $U^{\ast}_{n}$  the index of the first vertex of  $\Tn$ with maximal out-degree in the lexicographical order (starting from $0$).

 \begin{corollary}\label{cor:U-local}
 The following assertions hold as $n \rightarrow \infty$.
\begin{enumerate}
\item For every $x\in(0,1)$,
$\P \left( \frac{{\Lambda}( U^{\ast}_{n})}{{\Lambda}(n)} \leq x \right)  \rightarrow  x$.
\item  The convergence $ \frac{ 1}{n} U^{\ast}_{n}  \rightarrow  0$ holds in probability. 
\end{enumerate}
 \end{corollary}
  
 \begin{proof}
 By definition, $U^{\ast}_{n}$ is the index of the first maximal jump of $\W(\Tn)$. By \cref{prop:GWRW}, \cref{thm:dTVLoc} and \cref{cor:timereversal}, it is enough to establish the result when $U^{\ast}_{n}$ is replaced with $ \widehat{I}_{n-1}$. Since $ \widehat{I}_{n-1}$ and $I_{n-1}$ have the same law,  the result follows from \cref{prop:In}. 
 \end{proof}

  We now establish a limit theorem for the height $H^{\ast}_{  n}$   of the first vertex of $\Tn$ with maximal out-degree (\cref{thm:height-local}).

 \begin{proof}[Proof of \cref{thm:height-local}] Thanks to the relation between the height and the {\L}ukasiewicz path (see e.g.~\cite[Proposition 1.2]{LG05}) we have \[H^{\ast}_{  n}= \# \left\{0 \leq i <  U^{\ast}_{n} : \W_{i}(\Tn)=\min_{[ i,  U^{\ast}_{n} ] } \W(\Tn) \right\}.\] Recall that $V^z_{n}$ is the index of the first largest jump of $\Zn$. By \cref{prop:GWRW} and  \cref{thm:dTVLoc}  it is enough to establish that
  \[ \frac{1}{\Lambda(n)} \cdot \# \left\{0 \leq  i < V^z_{n}: \Zn_{i}=\min_{[ i, V^z_{n} ] } \Zn \right\}  \quad \xrightarrow[n\to\infty]{(d)} \quad \mathsf{Exp}(1).\]
 By \cref{cor:timereversal}, we can assume without loss of generality that  the maximal jump of  $(W_0,W_1,\ldots,W_{n-1},-1)$  is the last one, so that
 \[\# \left\{0 \leq i < V^z_{n}: \Zn_{i}=\min_{[ i, V^z_{n} ] } \Zn \right\}=\#  \left\{0 < i \leq \widehat{I}_{n-1}: \widehat{W}_{i}=\max_{[ 0,i ]} \widehat{W}\right\}.\]
Since $(\widehat{W}_{i} : 0 \leq i < n)$ and $(W_{i}: 0 \leq i < n)$ have the same distribution, and moreover $\P(T_{1}>n) \sim \frac{1}{\Lambda(n)}$ by \cref{prop:laddertimes} $(i)$, the desired result follows from \cref{prop:numberladdertimes}.\end{proof}

  \begin{remark}
   In particular, $U^{\ast}_{n} \rightarrow \infty$ and $H_{n}^{*} \rightarrow \infty$ in probability. However, if $\mu$ is subcritical and in the domain of attraction of a Cauchy distribution, $U^{\ast}_{n}$ and $H_{n}^{*}$ converge in distribution as $n \rightarrow \infty$ (this can be seen by adapting the arguments of \cite{KR18} together with the results of \cite{Ber17}).
   \end{remark}
   
We are now ready to prove Theorem \ref{thm:degrees-local}.
   
   \begin{proof}[Proof of Theorem \ref{thm:degrees-local}]
By \cref{prop:GWRW} and \cref{thm:dTVLoc}, it is enough to establish the result with $\W(\Tn)$ replaced with $\Zn= \mathcal{V}(W_0,W_1,\ldots,W_{n-1},-1)$. We keep the notation $V^z_{n}$ for the index of the first largest jump of $\Zn$, and work on the event $\mathcal{E}_n$ thanks to \cref{cor:timereversal}. 
 
Recall that $(\Zn_{i} : 0\leq i  \leq V^z_{n})=(\widehat{W}_{i} : 0\leq i  \leq \widehat{I}_{n-1})$ and that moreover $I_{n}/n \rightarrow 0$ in probability by Proposition \ref{prop:In}. Thus, by the functional convergence \eqref{eq:cvfonctionnelle} (applied with $\widehat{W}$ instead of $W$) and standard properties of Skorokhod's $J_{1}$ topology, we get that \[\frac{1}{a_{n}} \sup_{0 \leq i < V^z_{n}}{\left|\Zn_{i+1}-\Zn_{i}\right|}  \quad \xrightarrow[n \rightarrow \infty]{(\P)} \quad 0,\] meaning that the first $V^z_n-1$ jumps of $\Zn$ are $o(a_n)$.

By the proof of  \cref{thm:RW-local} and  Skorokhod's representation theorem we may assume that the following  convergences hold almost surely as $n\rightarrow \infty$:
\begin{equation}
\label{eq:cv1} \frac{V^z_{n}}{n} \rightarrow 0, \,\, \frac{1}{|b_{n}|} \sup_{[ 0,V^z_{n} ]}{|\Zn|} \rightarrow 0, \,\,   \frac{1}{a_{n}} \sup_{0 \leq i < V^z_{n}}{\left|\Zn_{i+1}-\Zn_{i}\right|} \rightarrow 0, \,\,  \frac{1}{|b_{n}|}\left(\Zn_{V^z_{n}+1}-\Zn_{V^z_{n}}\right)  \rightarrow 1
\end{equation}
and
\begin{equation}
\label{eq:cv2}\left( \frac{\Zn_{V^z_{n}+1+\lfloor nt \rfloor}-\Zn_{V^z_{n}+1}- b_{n} t}{a_{n}} : 0 \leq t \leq  1 \right)  \quad \xrightarrow[n\to\infty]{(d)} \quad (\mathcal{C}_{t})_{0 \leq t \leq 1},
\end{equation}
where we set $Z^{(n)}_{k}=0$ for $k>n$.

Therefore, for $n$ sufficiently large, we have $\Delta^{(0)}_{  n}=\Zn_{V^z_{n}+1}-\Zn_{V^z_{n}}$ and $(\Delta^{(i)}_{  n} : i \geq 1)$ are the jumps of $(\Zn_{i} : V^z_{n}+1 \leq i \leq {n})$ in decreasing order.
Since $ \mathcal{C}$ is almost surely continuous at $1$,  by continuity properties of the Skorokhod topology, we get that $ ( {\Delta^{(1)}_{  n} }/{a_{n}},{\Delta^{(2)}_{  n} }/{a_{n}}, \ldots )$ converges in distribution to the  decreasing rearrangement of the jumps of $ (\mathcal{C}_{t}, 0 \leq t \leq 1)$.
Since the Lévy measure of $ \mathcal{C}$ is $ \mathbbm{1}_{x>0} \frac{{\d}x}{x^{2}}$, the desired result follows from the fact that	$(\mathcal{C}_{t}- \mathcal{C}_{t-} , t \geq 0)$ is a Poisson point process with intensity $ \mathbbm{1}_{x>0} \frac{{\d}x}{x^{2}}$ (see e.g.~\cite[Section 1.1]{Ber96}).  \end{proof}

\section{Cauchy random walks: tail conditioning}
\label{sec:tail}

In this section, we deal with a $\BGW_\mu$ tree $\Tgn$ conditioned to have \textit{at least} $n$ vertices, when the offspring distribution $\mu$ satisfies the more general assumption \eqref{eq:hypmu}.

\smallskip

In order to do so, we consider a random walk $(W_i : i\geq 0)$ whose increments satisfy assumption \eqref{eq:hypX}. Contrary to Section \ref{sec:local}, we aim at studying the behaviour of the \emph{meander} $(\Wgn_{i} : i \geq 0)$, defined as $(W_{i}: i \geq 0)$ under the ``tail'' conditional probability $ \P( \, \cdot \, | \zeta \geq n)$, which is the {\L}ukasiewicz path of $\Tgn$. (We use the notation $\Wgn$ because the notation $W^{(n)}$ has been used when working under the local conditioning $ \P( \, \cdot \, | \zeta = n)$).

More precisely, we shall couple with high probability the trajectory $(\Wgn_{i}: i \geq 0)$ with that of a random walk conditioned to be nonnegative for a random number of steps (whose number converges in probability to $\infty$ as $n \rightarrow \infty$), followed by an independent ``big jump'', and then followed by an independent unconditioned random walk. 

We will use again the notation and results of Section \ref{sec:Estimates}. 

\subsection{Invariance principle for meanders}
\label{ss:notation}

First recall that $I_{n}$ is the last weak ladder time of $(W_{i}: 0 \leq i \leq n)$. For $n \geq 1$, we consider the process $(\Zgn_{i} : i \geq 0)$ whose distribution is specified as follows. 

For every $j \geq 1$, conditionally given $\{I_{n}=j-1\}$, the three random variables $(\Zgn_{i} : 0\leq i <j)$, $\Ygn_{j} \coloneqq \Zgn_j-\Zgn_{j-1}$ and $(\Zgn_{i+j}-\Zgn_j : i \geq 0)$ are independent and distributed as follows:
\begin{itemize}
	\item $\displaystyle (\Zgn_{i} : 0\leq i <j) \, \mathop{=}^{(d)}  \, (W_{i} : 0\leq i <j) \textrm{ under } \P(\, \cdot \, | \zeta \geq j)$
	\item $\displaystyle \Ygn_{j}  \,  \mathop{=}^{(d)} \, X \textrm{ under } \P(\, \cdot \, | X \geq |b_n|)$
	\item $\displaystyle (\Zgn_{i+j}-\Zgn_j : i \geq 0) \, \mathop{=}^{(d)}   \, (W_{i} : i \geq 0)$.
\end{itemize}

Our main result is the following.

\begin{theorem}\label{thm:dTVtail} We have
\[d_{\mathrm{TV}}\left( \left(\Wgn_{i} : i \geq 0 \right),  \left(\Zgn_{i} : i \geq 0 \right) \right)  \quad \xrightarrow[n\to\infty]{}  \quad 0,\]
where $d_{\mathrm{TV}}$ denotes the total variation distance on $\R^{\Z_{+}}$ equipped with the product topology.
\end{theorem}

 Intuitively speaking, this means that under the conditional probability $ \P( \, \cdot \, | \zeta \geq n)$, as $n \rightarrow \infty$, the random walk $(W_{i} : i \geq 0)$ first behaves as conditioned to stay nonnegative for a random number $I_{n}$ of steps, then makes a jump distributed as  $ \P( \, \cdot \, | X \geq |b_n|)$, and finally evolves as a non-conditioned walk. See \cref{prop:In} above for an estimate on the order of magnitude of~$I_{n}$.

\begin{example}When $\mu(n) \sim \frac{c}{n^{2} \ln(n)^{2}}$ as $n \rightarrow \infty$, by \cref{prop:In} and  Example \ref{ex:1} we have that $ \frac{\ln(I_{n})}{\ln(n)}$ converges in distribution to a uniform random variable $U$ on $[0,1]$. In other words, the time $I_{n}$ of the ``big jump'' of $\Zgn$ is of order $n^{U}$.
\end{example}

\paragraph*{Proof of \cref{thm:dTVtail}}
The structure of the proof is similar to that of \cite[Theorem 7]{KR18}. However, in the latter reference, $I_{n}$ converges in distribution as $n \rightarrow \infty$ to an integer-valued distribution, while here $I_{n} \rightarrow \infty$ in probability. For these reasons, the approach is more subtle.

Let us introduce some notation. Let $\cA$ be the Borel $\sigma$-algebra on $\R^\N$ associated with the product topology and set
\[\mu_n(A) \coloneqq \Pr{\left(\Wgn_i : i \geq 0\right)\in A} \quad \text{and} \quad \nu_n(A) \coloneqq \Pr{\left(\Zgn_i : i \geq 0\right)\in A}, \quad A\in \cA.\] The idea of the proof of \cref{thm:dTVtail} is to transform the estimates of \cref{cor:utile} into an estimate on probability measures by finding a ``good" event $G_n$ such that $\nu_{n}(G_{n}) \rightarrow 1$ as $n \rightarrow \infty$ and then by showing that $\sup_{A \in \cA }{\vert \mu_n(A\cap G_{n})-\nu_n(A\cap G_{n})\vert} \rightarrow 0$ as $n \rightarrow\infty$.

By Proposition \ref{prop:In}, we have that $I_{n}/n$ converges in probability to $0$ as $ n \rightarrow \infty$. As a consequence, we may find a sequence $(x_{n} : n\geq 1)$ such that $x_{n}=o(n)$ and $\P(I_{n} \geq  x_{n}) \rightarrow 0$. From now on, we let $(x_{n} : n\geq 1)$ be such a sequence.

\begin{lemma}\label{lem:PFCauchy} For every $n\in\N$, set 
\[
	G_{n}:=\left\lbrace (w_0,\ldots,w_n) \in \R_+^{n+1} : \exists! \ i \in \llbracket 1,x_n\rrbracket \text{ s.t. } w_i-w_{i-1} \geq \vert b_n \vert \right\rbrace.
\] Then, $\nu_{n}( G_{n} ) \longrightarrow 1$ as $n \rightarrow \infty$.\end{lemma}

Let us first explain how one establishes  \cref{thm:dTVtail} using  \cref{lem:PFCauchy}.

\begin{proof}[Proof of \cref{thm:dTVtail}]By Lemma \ref{lem:PFCauchy}, it suffices to show that, as $n \rightarrow\infty$, we have  $\sup_{A \in \cA }{\vert \mu_n(A\cap G_n)-\nu_n(A\cap G_n)\vert} \rightarrow 0$. Without loss of generality, we  focus on events of the form $\mathbf{w}\times A$, where $\mathbf{w}=(0,w_1,\ldots,w_n)\in G_n$ and $A\in \cA$. 

On the one hand, since $\mathbf{w} \in G_{n}$ we have
\[\mu_n\left(\mathbf{w}\times A\right)=\frac{\Pr{(W_{0},W_{1}, \ldots,W_{n})=\mathbf{w}} \Pr{\left(W_{n+i} : i \geq 1\right) \in A }}{\Pr{\zeta \geq n}}.\] On the other hand,  write
\[\nu_n\left(\mathbf{w}\times A\right)=\sum_{j=1}^\infty \Pr{I_n=j, \ \left(\Zgn_i : i \geq 0\right) \in \mathbf{w}\times A}.\] Since $\mathbf{w}\in G_n$, there is a unique value of   $j \in \llbracket 1,x_n\rrbracket$ such that $w_j-w_{j-1} > \vert b_n\vert$, which we denote by $j(\mathbf{w})$. Hence
\begin{align*}
& \Pr{I_n< x_n, \ \left(\Zgn_i : i \geq 0\right) \in \mathbf{w}\times A}\\
	&=\sum_{j=1}^{x_n} \Pr{I_n=j-1}\cdot  \frac{\Pr{(W_{0}, \ldots,W_{j-1})=(w_0,\ldots,w_{j-1})}}{\Pr{\zeta \geq j}} \cdot \frac{\Pr{X =w_j-w_{j-1}, \ X \geq \vert b_n\vert }}{\Pr{X \geq \vert b_n\vert }}  \\
&	\hspace{2.3cm}\qquad\qquad\qquad\qquad\qquad\qquad\qquad\cdot \Pr{\left(W_{i+j} : i \geq 1\right) \in (w_{j+1},\ldots,w_{n}) \times A }\\
 &=\frac{\Pr{I_n=j(\mathbf{w})-1}}{\Pr{\zeta \geq j(\mathbf{w})}\Pr{X \geq \vert b_n\vert }}\cdot \Pr{(W_{0},W_{1}, \ldots,W_{n})=\mathbf{w}}\Pr{\left(W_{n+i} : i \geq 1\right) \in A }.
\end{align*}
We therefore obtain by Lemma \ref{lem:In} that
\begin{align*}
\vert \mu_n(\mathbf{w}\times A)-\nu_n(\mathbf{w}\times A)\vert &\leq \Pr{I_n\geq x_n} + \left|\frac{\Pr{I_n=j(\mathbf{w})-1}\Pr{\zeta\geq n}}{\Pr{\zeta \geq j(\mathbf{w})}\Pr{X \geq \vert b_n\vert }}-1\right|\\
&\leq  \Pr{I_n\geq x_n} + \left|\frac{\Pr{T_1>n-j(\mathbf{w})+1}\Pr{\zeta\geq n}}{\Pr{X \geq \vert b_n\vert }}-1\right|.	
\end{align*}
The first term goes to zero as $n \rightarrow \infty$ by definition of $(x_n : n \geq 1)$, as well as the second one since by \cref{cor:utile}, we have  $ \P(T_{1}>n-j+1) \Pr{\zeta \geq n} \sim \Pr{X \geq |b_n|}$ uniformly in $1 \leq j \leq x_{n}$. This completes the proof.\end{proof}

\begin{proof}[Proof of \cref{lem:PFCauchy}]
First, set $\eta_{n}= \sqrt{a_{n} |b_{n}|}$ and recall that $\Yn_{j} = \Zgn_j-\Zgn_{j-1}$ for every $j \geq 1$. Then, observe that the event
\begin{align*}
\{I_n \leq x_n\} \cap \left\{ \max_{1 \leq i \leq I_n} \Yn_{i}< \vert b_n\vert \right\} &\cap \left\{ \Yn_{I_n+1}> \vert b_n\vert + \eta_n \right\}  \cap\left\{ \max_{I_{n}+1 < i \leq x_{n}} \Yn_{i}< \vert b_n\vert \right\}
\\
&\cap \left\{  \min_{0 \leq i \leq n-I_n} \left(\Zgn_{I_n+i}-\Zgn_{I_n}\right) > -\vert b_n\vert -\eta_{n} \right\}
 \end{align*}
is included in the event $\{(\Zgn_0, \ldots, \Zgn_n )\in G_{n}\}$. As a consequence, $\nu_{n}(\overline{G_{n}})$ is bounded from above by
\begin{align*}
	   \Pr{I_n>x_n} &+ \max_{1\leq j \leq x_n}\Pr{\max_{1\leq i \leq j}X_i \geq \vert b_n\vert \ \Big| \ \zeta \geq j} +\Pr{X < \vert b_n\vert + \eta_n \mid X \geq \vert b_n\vert }\\
		&  + \Pr{\max_{1 \leq i \leq n} X_{i} \geq \vert b_{n}\vert}+  \Pr{ \min_{1\leq i \leq n}{W_i}\leq -\vert b_n\vert-\eta_{n}}.
\end{align*} Since $\P(I_n>x_n) \rightarrow 0$ as $n \rightarrow \infty$, it is enough to show that each one of the four last terms of the above inequality tends to $0$ as $n \rightarrow \infty$.

\paragraph*{First term.}  Let us show that $ \Pr{\max_{1 \leq i <j}{X_i} \geq |b_n| \ \middle| \ \zeta\geq j}  \rightarrow 0$ uniformly in $1 \leq j \leq x_{n}$. To this end, by decomposing the event $ \{\max_{1 \leq i <j}{X_i} \geq |b_n|\}$ according to the position of the first jump greater than $|b_n|$, write
\begin{align*}
 \Pr{\max_{1 \leq i <j}{X_i} \geq |b_n| \ \middle| \ \zeta\geq j} \leq \frac{1}{\P(\zeta \geq j)} \sum_{k=1}^{j-1} \P(X \geq |b_n|) \P(\zeta \geq k).
 \end{align*}
Hence it remains to check that
\[ \frac{\P(X \geq b_{n})}{\P( \zeta \geq x_{n})} \sum_{k=1}^{x_{n}} \P(\zeta \geq k)  \quad \xrightarrow[n\to\infty]{} \quad 0.\]
But since $\P(I_{n} \leq x_{n}) \rightarrow 1$ as $n\rightarrow \infty$, we know by Lemma \ref{lem:In} that 
\[\sum_{k=0}^{x_{n}-1} \P(\zeta > k) \P(T_{1}>n-k) =\sum_{k=1}^{x_{n}-1} \P(I_{n}=k)  \quad \xrightarrow[n\to\infty]{} \quad  1.\]
Since $x_{n}=o(n)$, by \cref{cor:utile} $(i)$ we have $\P(T_{1}>n-k)\sim \P(T_{1}>n)$ as $n \rightarrow \infty$, uniformly in $1 \leq k \leq x_{n}$. Therefore, $\sum_{k=1}^{x_{n}}\P(\zeta \geq k) \sim \frac{1}{\P(T_{1}>n)}$.  Hence, by \cref{cor:utile} $(iii)$,
\[ \frac{\P(X \geq |b_{n}|)}{\P( \zeta \geq x_{n})} \sum_{k=1}^{x_{n}} \P(\zeta \geq k)  \quad \mathop{\sim}_{n \rightarrow \infty} \quad  \frac{\P(\zeta>n)}{\P(\zeta \geq x_{n})}.\]
Since $x_{n}=o(n)$, this term tends to $0$ as $n \rightarrow \infty$  by \cref{cor:utile} $(ii)$.
\medskip

\paragraph*{Second term.}
Write \begin{equation*}
	 \Pr{ X < {|b_n|} + \eta_{n} \ \middle| \ X \geq {|b_n|}} =1- \frac{\P(X \geq  |b_n| + \eta_{n})}{\P(X \geq |b_n|)}=  1-\frac{L(|b_n|+\eta_{n})}{L(|b_n|)}   \frac{1}{1+\eta_{n}/|b_n|}
	 \end{equation*}
which tends to $0$ as $n \rightarrow \infty$ since ${\eta_{n}}/{|b_n|}=\sqrt{ {a_{n}}/{|b_n|}}\rightarrow 0$.

\paragraph*{Third term.} Using \eqref{eq:bn}, write 
\[\Pr{\max_{1 \leq i \leq n} X_{i} \geq \vert b_{n}\vert} \leq n \Pr{X_{1} \geq \vert b_{n} \vert}= \frac{nL(\vert b_{n}|)}{\vert b_{n} \vert}  \quad \mathop{\sim}_{n \rightarrow \infty} \quad  \frac{L(\vert b_{n}\vert)}{\ell^{*}(\vert b_{n} \vert)},\]
which tends to $0$ as $n \rightarrow \infty$ since $\ell^{*}(n)/L(n) \rightarrow \infty$.

\paragraph*{Fourth term.}
Write
\[
\Pr{ \min_{1\leq i \leq n}{W_i}\leq - |b_{n}|-\eta_{n}} \leq \Pr{ \inf_{0 \leq t \leq 1} (W_ {\lfloor nt \rfloor} - b_{n} t)\leq -\eta_{n}} =\Pr{ \inf_{0 \leq t \leq 1}  \frac{W_ {\lfloor nt \rfloor} - b_{n} t}{a_{n}} \leq - \frac{\eta_{n}}{a_{n}}}.
\]
By \eqref{eq:cvfonctionnelle}, since the infimum is a continuous functional on $\D([0,1])$ (see e.g.~\cite[Chapter~VI, Proposition~2.4]{JS03},  $ \inf_{0 \leq t \leq 1}  \tfrac{W_ {\lfloor nt \rfloor} - b_{n} t}{a_{n}} $ converges in distribution to a real-valued random variable and since ${\eta_{n}}/{a_{n}}= \sqrt{{|b_{n}|}/{a_{n}}} \rightarrow \infty$, the last term indeed tends to $0$. \end{proof}

 We now establish a functional invariance principle for $\Wgn$, whose law we recall to be that of the random walk $(W_{i} : i \geq 0)$ under the conditional probability   $ \P( \, \cdot \, | \zeta \geq n)$.

\begin{theorem}
\label{thm:RW-tail}
   Let  $J$ be the real-valued random variable such that $\Pr{J \geq x}= 1/x$ for $x \geq 1$. Then, the convergence
\[ \left(  \frac{\Wgn_{\lfloor n t \rfloor}}{|b_{n}|} : t \geq -1 \right)   \quad \xrightarrow[n\to\infty]{(d)} \quad    \left( (J- t) \mathbbm{1}_{t \geq 0} : t \geq -1\right)\] holds in distribution in $\D([-1,\infty),\R)$.
In addition, the convergence
\[\frac{ \inf \{i \geq 1: \Wgn_{i}=-1\} }{|b_{n}|}   \quad \xrightarrow[n\to\infty]{(d)} \quad  J\] holds jointly in distribution.
  \end{theorem}
  
As in \cref{thm:RW-local}, we work with $\D([-1,\infty),\R)$ instead of  $\D(\R_{+},\R)$ since our limiting process almost surely
takes a positive value in $0$, while $\Wgn$ stays small for a positive time.

\begin{proof} The proof is similar to that of \cref{thm:RW-local}.  By \cref{thm:dTVtail}, it is enough to establish the result with $\Wgn$ replaced with $\Zgn$. By \cref{prop:In} $(ii)$, we have that $I_{n}/n \rightarrow 0$ in probability, while $\Ygn_{I_{n}+1}/|b_{n}| \rightarrow J $ in distribution as $n \rightarrow \infty$. Thus, it suffices to show that as $n \rightarrow \infty$,
\begin{enumerate}
\item  $  \frac{1}{|b_{n}|} \sup_{[ 0,I_{n} ]}{|\Zgn|} \rightarrow 0$ in probability,
\item  $\left(  \frac{1}{|b_{n}|} ({\Zgn_{I_{n}+1+\lfloor n t \rfloor}}-\Zgn_{I_{n}+1}) : t \geq 0 \right) \rightarrow    \left( - t  : t \geq 0\right)$  in distribution in $\D(\R_{+},\R)$. 
\end{enumerate}

Since by construction $(\Zgn_{I_{n}+1+i}-\Zgn_{I_{n}+1} : i \geq 0)$ has the same distribution as $(W_{i} : i \geq 0)$, the second assertion simply follows from the functional convergence \eqref{eq:cvfonctionnelle} combined with the fact that $a_{n}=o(|b_{n}|)$ and standard properties of Skorokhod's topology.

For the first assertion, we use a time-reversal technique. By arguing as in the proof of  \cref{lem:In}, we see that
\begin{equation}
\label{eq:timereversal}(\Zgn_{i} : 0\leq i  \leq I_{n}) \quad  \mathop{=}^{(d)}  \quad  \left(W^{[I_{n}]}_{i} : 0\leq i  \leq I_{n}\right),
\end{equation}
where  $W^{[I_{n}]}_{i} \coloneqq W_{I_{n}}-W_{I_{n}-i}$ for  $0 \leq i \leq I_{n}$.
Hence, by combining \eqref{eq:cvfonctionnelle} with the fact that $I_{n}/n \rightarrow 0$ in probability and $a_{n}=o(|b_{n}|)$, we get that   $\frac{1}{|b_{n}|} \sup_{[ 0,I_{n} ]}{|W^{[I_{n}]}|} \rightarrow 0$ in probability. Since $ \sup_{[ 0,I_{n} ]}{|\Zgn|}$ is stochastically bounded from above by $2\sup_{[ 0,I_{n} ]}{|W^{[I_{n}]}|}$, we obtain the desired result.
\end{proof}

\subsection{Application: limit theorems for $\BGW$ trees}\label{ss:AppliGWqueue}

From now on, we let $\mu$ be an offspring distribution satisfying \eqref{eq:hypmu}, and let $ \Tgn$ be a $ \BGW_{ \mu}$ tree conditioned on having at least $n$ vertices. The goal is to apply the results of the previous sections to the study of the tree $\Tgn$.

Here, our strategy is very similar to that of Section \ref{ss:AppliGWloc}, where we replace Theorem \ref{thm:dTVLoc} by Theorem \ref{thm:dTVtail}. For this reason, we give less detailed proofs. For instance, Theorem \ref{thm:loop-tail} is proved along the same lines as Theorem \ref{thm:loop-local} and is simpler, so we omit the details. Next, we immediately obtain a limit theorem for the {\L}ukasiewicz path $\W(\Tgn)$ by simply combining \cref{prop:GWRW} with \cref{thm:RW-tail}. As before, \cref{thm:dTVtail} gives a simple and efficient way to asymptotically simulate~$\Tgn$.

\begin{proposition}
  \label{prop:luka-tail}
    Let  $J$ be the real-valued random variable such that $\Pr{J \geq x}= 1/x$ for $x \geq 1$. Then the convergence
\[ \left(  \frac{\W_{\lfloor n t \rfloor}(\Tgn)}{|b_{n}|} : t \geq -1 \right)   \quad \xrightarrow[n\to\infty]{(d)} \quad    \left( (J- t) \mathbbm{1}_{t \geq 0} : t \geq -1\right)\] holds in distribution in $\D([-1,\infty),\R)$.
In addition, the convergence
\[\frac{|\Tgn|}{|b_{n}|}   \quad \xrightarrow[n\to\infty]{(d)} \quad  J\] holds jointly in distribution.
  \end{proposition}

We now obtain information concerning the vertex with maximal degree in $\Tgn$. Using \cref{prop:laddertimes}, we let $\Lambda$ be the increasing slowly varying function such that
\[ \P(\inf \{j > 0: W_{j} \geq 0\} \geq n)=\frac{1}{\Lambda(n)}.\]
We also denote by $U^{\ast}_{\geq n}$ the index in the lexicographical order (starting from $0$) of the first vertex of $\Tn$ with maximal out-degree.

 \begin{corollary}\label{cor:U-tail}
 The following assertions hold as $n \rightarrow \infty$.
\begin{enumerate}
\item For every $x\in(0,1)$,
$\P \left( \frac{{\Lambda}( U^{\ast}_{ \geq n})}{{\Lambda}(n)} \leq x \right)  \rightarrow  x$.
\item  The convergence $ \frac{ 1}{n} U^{\ast}_{ \geq n}  \rightarrow  0$ holds in probability. 
\end{enumerate}
 \end{corollary}
 
 \begin{proof}
 By \cref{prop:GWRW} together with Theorems \ref{thm:dTVtail} and \ref{thm:RW-tail}, it is enough to establish the result when $U^{\ast}_{ \geq n}$ is replaced with ${I}_{n}$. It then follows from \cref{prop:In}.
 \end{proof}

 Next, we  establish a limit theorem for the height $H^{\ast}_{ \geq n}$   of the first vertex of $\Tgn$ with maximal out-degree (\cref{thm:height-tail}).
 
 \begin{proof}[Proof of \cref{thm:height-tail}]
 By arguing as in the proof of \cref{thm:height-local} and using again Theorems \ref{thm:dTVtail} and \ref{thm:RW-tail}, it is enough to show the result when $H^{\ast}_{ \geq n}$ is replaced with \[\# \left\lbrace 0 \leq i < I_{n} : \Zgn_{i}=\min_{[ i, I_{n} ] } \Zgn \right\rbrace.\] Thanks to the time-reversal identity \eqref{eq:timereversal}, the desired result then follows from \cref{prop:numberladdertimes}.\end{proof}

\begin{remark}
The conclusions of \cref{cor:U-local} and  \cref{thm:height-local} (for $\Tn$) are respectively the same as those of \cref{cor:U-tail} and \cref{thm:height-tail} (for $\Tgn$). This may be alternatively explained by the following fact: on an event of high probability, $\Tn$ and $\Tgn$ have the same distribution once one removes the descendants of the vertex with maximal degree (since we do not require this stronger statement, we do not give a proof).
\end{remark}

We conclude by establishing Theorem \ref{thm:degrees-tail}.
   
   \begin{proof}[Proof of Theorem \ref{thm:degrees-tail}] The proof follows that of \cref{thm:degrees-local}. First, by \cref{thm:dTVtail}, we may replace $\W(\Tgn)$ with $\Zgn$.  Then, as in the proof of \cref{thm:degrees-local}, we observe that the first $I_{n}$ jumps of $\Zgn$ are $o(a_{n})$, and that by Skorokhod's representation theorem, one may assume that the following  convergences hold almost surely
 \begin{equation}
\label{eq:cv1B}\frac{1}{|b_{n}|} \sup_{[ 0,I_{n}]}{\left|\Zgn\right|} \rightarrow 0, \quad   \frac{1}{a_{n}} \sup_{0 \leq i < I_{n}}{\left|\Zgn_{i+1}-\Zgn_{i}\right|} \rightarrow 0, \quad  \frac{1}{|b_{n}|}\Ygn_{I_{n}+1}  \rightarrow J
\end{equation}
and
\begin{equation}
\label{eq:cv2B}\left( \frac{\Zgn_{I_{n}+1+\lfloor nt \rfloor}-\Zgn_{I_{n}+1}- b_{n} t}{a_{n}} : t \geq 0 \right)  \quad \xrightarrow[n\to\infty]{(d)} \quad (\mathcal{C}_{t})_{t \geq 0}.
\end{equation}
   
Therefore, for $n$ sufficiently large, we have $\Delta^{(0)}_{ \geq n}=\Ygn_{I_{n}+1}$ and $(\Delta^{(i)}_{ \geq n})_{i \geq 1}$ are the jumps of $(\Zgn_{i} : I_{n} < i \leq \zeta_{n})$ in decreasing order, where $\zeta_{n}=\inf \{i \geq 1;\Zgn_{i} =-1 \}$. 
Since $a_{n}=o(|b_{n}|)$, by \eqref{eq:cv1B} and \eqref{eq:cv2B}, $ \frac{1}{n}\zeta_{n} \rightarrow J$ and we conclude as in the proof of \cref{thm:degrees-local}.\end{proof}

\section{Application to random planar maps}\label{sec:Maps}
 
We now deal with an application of \cref{thm:loop-tail} to the study of the boundary of \textit{Boltzmann maps}. A (planar) map is the proper embedding of a finite connected graph into the 2-dimensional sphere, seen up to orientation-preserving homeomorphisms. In order to break symmetries, we assume that maps carry a distinguished oriented \emph{root} edge. The faces of a map are the connected components of the sphere deprived of the embedding of the edges, and the degree of this face is the number of its incident oriented edges. Given a weight sequence $\q=(q_1,q_2,\ldots)$ of nonnegative real numbers, the Boltzmann weight of a \textit{bipartite} map $\m$ (i.e.\ whose faces have even degree) is given by
\[w_\q(\m)\coloneqq\prod_{f \in \textup{Faces}(\m)}q_{\deg(f)/2}.\] The sequence $\q$ is termed \textit{admissible} if these weights form a finite measure on the set of bipartite maps. Then, a $\q$-Boltzmann map is a random planar map chosen with probability proportional to its weight. 

Over the years, a classification of weight sequences has emerged in the literature, following the milestones laid in \cite{MM07,BBG12}. Besides admissibility, we usually assume that the weight sequence $\q$ is \textit{critical}, meaning that the expected squared number of vertices of the $\q$-Boltzmann map is infinite. Among critical weight sequences, further distinction is made by specifying the distribution of the degree of a \textit{typical face} of the $\q$-Boltzmann map. A critical weight sequence is called \textit{generic critical} if the degree of a typical face has finite variance, and \textit{non-generic critical} with parameter $\alpha\in(1,2)$ if the degree of a typical face falls within the domain of attraction of a stable law with parameter $\alpha$ (see \cite{MM07,LGM09} for more precise definitions).

This classification is justified by scaling limit results for Boltzmann maps conditioned to have a large number of faces. After the seminal papers \cite{Marckert2006,LG07}, Le Gall \cite{LG11} and Miermont \cite{Mie11} proved that uniform quadrangulations have a scaling limit, the \textit{Brownian map}. This convergence was later extended to generic critical sequences in \cite{Mar16}, building on the earlier works \cite{MM07,LG11}. In 2011, Le Gall and Miermont \cite{LGM09} established the subsequential convergence of \textit{non-generic critical} Boltzmann maps. The natural candidate for the limit is called the \textit{stable map} with parameter $\alpha$ (see \cite{Mar18b} for extensions allowing slowly varying corrections).

The geometry of the stable maps is dictated by large faces that remain present in the scaling limit. Predictions originating from theoretical physics suggest that their behaviour differ in the dense phase $\alpha\in(1,3/2)$, where they are supposed to be self-intersecting, and in the dilute phase $\alpha\in(3/2,2)$, where it is conjectured that they are self-avoiding. The strategy initiated in \cite{Ric17} and carried on in \cite{KR18} touches upon this conjecture via a discrete approach. It consists in studying Boltzmann maps \textit{with a boundary}, meaning that the face on the right of the root edge is viewed as the boundary $\partial\m$ of the map $\m$. As a consequence, this face receives unit weight and its degree is called the \textit{perimeter} of the map. Then, for every $k\geq 0$, we let $\Mc_{\geq k}$ be a $\q$-Boltzmann map conditioned to have perimeter larger than $2k$, so that its boundary $\partial\Mc_{\geq k}$ stands for a typical face of degree larger than $2k$ of a $\q$-Boltzmann map. 

The key observation of \cite[Corollary 3.7 \& Lemma 4.1]{Ric17} is that the random graph $\partial\Mc_{\geq k}$ can be described as $\Loop(\Tc_{\geq 2k+1})$, where $\Tc_{\geq 2k+1}$ is a $\BGW_\nu$ tree conditioned on having at least $2k+1$ vertices. The offspring distribution $\nu$ of this tree has been analyzed in \cite[Lemma 3.5 \& Proposition 3.6]{Ric17}. In the dense regime $\alpha\in(1,3/2)$, it was shown that $\nu$ is critical and heavy-tailed, so that the scaling limit of the boundary of Boltzmann maps conditioned to have large (fixed) perimeter is a so-called \textit{random stable looptree} introduced in \cite{CK14b}. On the contrary, in the dilute phase $\alpha\in(3/2,2)$, $\nu$ is subcritical and heavy-tailed, and \cite[Corollary 4]{KR18} shows that the scaling limit of $\partial\Mc_{\geq k}$ when $k$ goes to infinity is a circle with random perimeter. Together, these results show the existence of a phase transition on the geometry of large faces at $\alpha=3/2$. 

The purpose of the following application is to discuss the critical case $\alpha=3/2$. It was established in \cite[Lemma 6.1]{Ric17} that in this case, the offspring distribution $\nu$ can be either subcritical or critical. However, the aforementioned predictions from theoretical physics (see Remark \ref{rem:CLE}) suggest that the scaling limit should be a circle in both cases. Moreover, it is conjectured in \cite{Ric17} that the offspring distribution $\nu$ falls within the domain of attraction of a Cauchy distribution. However, due to technical difficulties involving analytic combinatorics, this was only established in \cite[Proposition 6.2]{Ric17} for a specific weight sequence $\q^*$ defined by \begin{equation}\label{eqn:Qstar}
	q^*_k:=\frac{1}{4 }6^{1-k} \frac{\Gamma(k-3/2)}{\Gamma(k+5/2)}\mathbf{1}_{k\geq 2} \quad k\in \N.
\end{equation} This weight sequence was first introduced in \cite{ABM16} (see also \cite[Section 5]{BC17}). It turns out that $\q^*$ is non-generic critical with parameter $3/2$, and \cite[Proposition 6.2]{Ric17} entails that the associated offspring distribution $\nu$ is critical and satisfies \[\nu([k,\infty))\underset{k\rightarrow \infty}{\sim} \frac{1}{k\ln^2(k)}.\] A direct application of Theorem \ref{thm:loop-tail} gives the following result.

\begin{corollary}
\label{cor:ScalingNonGeneric} For every $k\geq 0$, let $\Mc_{\geq k}$ be a Boltzmann map with weight sequence $\q^*$ conditioned to have perimeter at least $2k$. Let  $J$ be the real-valued random variable such that $\Pr{J \geq x}= 1/x$ for $x \geq 1$. Then, there exists a slowly varying function $L^*$ tending to infinity such that the convergence
	\[\frac{L^*(k)}{k}\cdot \partial \Mc_{\geq k} \quad \xrightarrow[k\to\infty]{(d)} \quad J \cdot \mathbb{S}_1\] holds in distribution for the Gromov--Hausdorff topology.
\end{corollary}

As mentioned above, we believe that this result holds in greater generality, namely for all non-generic critical weight sequences with parameter $3/2$.

\begin{remark}\label{rem:CLE} Part of the motivation for this result comes from a stronger form of the celebrated Knizhnik--Polyakov--Zamolodchikov (KPZ) formula \cite{KPZ88} that we briefly describe. On the one hand, it is conjectured that planar maps equipped with statistical mechanics models converge towards a so-called \textit{Liouville Quantum Gravity} (LQG) surface \cite{DS11} coupled with a \textit{Conformal Loop Ensemble} (CLE) of a certain parameter $\kappa \in (8/3,8)$ (which is a random collection of loops, see \cite{She09,SW12}). On the other hand, non-generic critical Boltzmann maps are related to maps equipped with an $\mathcal{O}(n)$ loop model \cite{BBG12} through the \textit{gasket decomposition}. As a consequence, there is a conjectural relation between the parameter $\alpha \in (1,2)$ of Boltzmann maps and the parameter $\kappa$ of CLEs, given by the formula 
\[\alpha=\frac{1}{2}+\frac{4}{\kappa}.\] In this correspondence, the faces of the map play the role of loops of CLEs. It is also proved in \cite{RS05} that CLEs admit a phase transition between a dense, self-intersecting phase and a dilute self-avoiding phase at $\kappa=4$. Through this correspondance, $3/2$-stable maps are thus related to $\textup{CLE}_4$. Although self-avoiding, $\textup{CLE}_4$ loops are ``very close from each other'' (see for instance the discussion in \cite[Section 1.1]{MSW17}). In our wording, this critical phenomenon corresponds to the fact that the scaling limit of large faces in non-generic critical Boltzmann maps with parameter $3/2$ is still a circle, but with a renormalizing sequence that is possibly $o(n)$, in sharp contrast with the dilute regime. \end{remark}

\begin{remark} The condensation principle established in Theorem \ref{thm:loop-local} should also have an application to the study of non-generic critical Boltzmann maps with parameter $\alpha=1$ (i.e.\ such that the degree of a typical face is in the domain of attraction of a stable law with parameter $1$). We believe that by using the argument of \cite{JS15}, the scaling limit of such maps should be the Brownian tree. This will be investigated in future work.\end{remark}
  
\bibliographystyle{alpha}

\end{document}